\documentclass[12pt]{amsart}

\usepackage[english]{babel}
\usepackage{amsmath,amssymb,amsfonts,amsthm,amsopn,amstext,amsxtra,amscd}
\usepackage{bm,mathrsfs,mathtools}
\usepackage{cite}
\usepackage{graphicx}
\usepackage{float}
\usepackage{xcolor}
\usepackage{url}

\usepackage[colorlinks,linkcolor=blue,anchorcolor=blue,citecolor=blue,backref=page]{hyperref}
\hypersetup{breaklinks=true}

\usepackage[norefs,nocites]{refcheck}  

\def\({\left(}
\def\){\right)}

\newtheorem{thm}{Theorem}
\newtheorem{theorem}[thm]{Theorem}
\newtheorem{lem}[thm]{Lemma}
\newtheorem{lemma}[thm]{Lemma}

\newtheorem{cor}[thm]{Corollary}

\numberwithin{equation}{section}
\numberwithin{thm}{section}
\numberwithin{table}{section}

\usepackage{todonotes}
\DeclareMathOperator{\supp}{supp}

\DeclareMathOperator{\SL}{SL}

\DeclareMathOperator{\vol}{vol}

\newcommand{\jac}[2]{\(\frac{#1}{#2}\)}

\newcommand{\mand}{\qquad\text{and}\qquad}

\newcommand{\Rform}{\mathcal R}
\newcommand{\distZ}[1]{\left\|#1\right\|_{\mathbb Z}}
\newcommand{\one}{\mathbf 1}

\newcommand{\balpha}{\bm{\alpha}}

\newcommand{\bxi}{\boldsymbol{\xi}}

\newcommand{\bell}{\boldsymbol{\ell}}

\newcommand{\cC}{\mathcal C}
\newcommand{\cD}{\mathcal D}
\newcommand{\cE}{\mathcal E}

\newcommand{\cG}{\mathcal G}

\newcommand{\cK}{\mathcal K}

\newcommand{\cM}{\mathcal M}

\newcommand{\cS}{\mathcal S}
\newcommand{\cT}{\mathcal T}

\newcommand{\cW}{\mathcal W}

\newcommand{\fS}{\mathfrak S}

\newcommand{\fs}{\mathfrak s}

\newcommand{\C}{\mathbb C}

\newcommand{\Q}{\mathbb Q}
\newcommand{\R}{\mathbb R}
\newcommand{\Z}{\mathbb Z}

\newcommand{\e}{\mathbf e}

\renewcommand{\vec}[1]{\mathbf #1}

\newcommand{\va}{\vec{a}}
\newcommand{\vb}{\vec{b}}
\newcommand{\vc}{\vec{c}}
\newcommand{\vn}{\vec{n}}

\newcommand{\vt}{\vec{t}}
\newcommand{\vu}{\vec{u}}
\newcommand{\vv}{\vec{v}}
\newcommand{\vx}{\vec{x}}
\newcommand{\vy}{\vec{y}}
\newcommand{\vz}{\vec{z}}

\title{Arithmetic structure of $L_2$-norms of $\mathrm{SL}_2(\Z)$ matrices}

\author{Igor E. Shparlinski}
\address{School of Mathematics and Statistics, University of New South Wales, Sydney NSW 2052, Australia}
\email{igor.shparlinski@unsw.edu.au}
\author{Yixiu Xiao}
\address{School of Mathematical Sciences, Shanghai Jiao Tong University, 800 Dongchuan RD, 200240 Shanghai, China}
\email{yixiuxiao98@gmail.com}

\begin{document}
 
\begin{abstract}
For a matrix $\gamma\in\SL_2(\Z)$, we define 
\[
\Rform(\gamma)=a_1^2+a_2^2+a_3^2+a_4^2,
\qquad \text{where} \ 
\gamma=\begin{pmatrix}a_1&a_2\\ a_3&a_4\end{pmatrix},  
\]
and let $S_{\mathrm{sq}}(X)$ count the number of matrices $\gamma$ with 
$\|\gamma\|_\infty = \max\{|a_1|, |a_2|,|a_3|,|a_4|\} \leq X$ and such that $\Rform(\gamma)$ is squarefree.  We prove that
\[
S_{\mathrm{sq}}(X)
=
\fS_{\Rform}^{\mathrm{sq}}N(X)
+O(X^{19/10+o(1)}), \quad \text{as}\  X\to \infty,
\]
where
$N(X)=\#\{\gamma\in\SL_2(\Z):\|\gamma\|_\infty\leq X\}$ and
$\mathfrak S_{\Rform}^{\mathrm{sq}}$ is an explicit positive Euler product of
local $p^2$-densities.  
The proof combines the $\delta$-method for small moduli
 with a sum-of-two-squares estimate for large square divisors.
This complements a result of J.~B.~Friedlander and H.~Iwaniec (2009) on prime values of 
$\Rform(\gamma)$, which, however, is conditional on a very strong form of the Elliott--Halberstam conjecture.
We also show that $\Rform(\gamma)$ is squarefree and has at most $9$ prime divisors for at least $cN(X)/\log X$ matrices
$\gamma\in\SL_2(\Z)$ with $\|\gamma\|_\infty\le X$, where $c>0$ is an absolute constant.
\end{abstract}


\maketitle
\tableofcontents

\section{Introduction}
\subsection{Motivation and set-up} 
Let  
\[
\Gamma=\SL_2(\Z)
\]
and, for
\[
\gamma=\begin{pmatrix}a_1&a_2\\ a_3&a_4\end{pmatrix}\in\Gamma,
\]
put
\[
\Rform(\gamma)=a_1^2+a_2^2+a_3^2+a_4^2 \mand \|\gamma\|_\infty
=
\max\{|a_1|,|a_2|,|a_3|,|a_4|\}.
\]

For $X\ge1$, define
\begin{equation}
\label{eq:Gamma X}
\Gamma_X=\{\gamma\in\Gamma:~\|\gamma\|_\infty\le X\}
\mand
N(X)=\#\Gamma_X.
\end{equation}

By a classical result of Krieg~\cite{Krieg94}, we have
\begin{equation}
\label{eq:Krieg}
    N(X) = \frac{96}{\pi^2} X^2 + O(X \log X).
\end{equation}

Our work is motivated by a result of Friedlander and
Iwaniec~\cite[Theorem~2]{FrIw09}, conditional on a  strong form of the
Elliott--Halberstam conjecture, which gives matching upper and lower bounds for the number of matrices $\gamma\in\Gamma_X$
for which $\Rform(\gamma)$ is prime.  We
instead obtain an unconditional asymptotic formula for the number of matrices $\gamma\in\Gamma_X$ for which
$\Rform(\gamma)$ is squarefree.  
Thus we study
the counting function
\[
S_{\mathrm{sq}}(X)=
\#\{\gamma\in\Gamma_X:~\Rform(\gamma)\ \text{squarefree}\}.
\]
We note that Sedunova~\cite{Sed} gives such an asymptotic formula for counting squarefree values 
of $\Rform(\gamma)$ when matrices $\gamma$ are ordered by the $L^2$-norm, that is, for 
$\#\{\gamma\in\Gamma:~\Rform(\gamma) \le Y, \ \Rform(\gamma)\ \text{squarefree}\}$.The methods differ quite substantially:
Sedunova~\cite{Sed} uses  a modification of the argument of Friedlander and
Iwaniec~\cite{FrIw09}, whereas we employ the $\delta$-method of  Heath-Brown~\cite{HeathBrown96} in the form 
given by Str\"ombergsson,  S\"odergren and  Vishe~\cite{StrSoVi26}.  Sedunova~\cite{Sed}
gets a better level of distribution in her model of counting, while our approach is more general and
allows us to treat $\SL_2(\Z)$  matrices whose entries lie in translated rectangular boxes.
Furthermore, combining our distribution estimates of $\Rform(\gamma)$ with the weighted 
linear sieve in the form presented by  Greaves~\cite{Greaves01}, we show unconditionally that $\Rform(\gamma)$ is squarefree and has at most nine prime factors for at least $c N(X)/\log X$ matrices $\gamma\in\Gamma_X$, where $c>0$ is an absolute constant.

\subsection{Main results}
To state our first result, we set
\[
\beta_{2^2}=0,
\]
and, for an odd prime $p$, define
\begin{equation}
\label{eq:beta-p}
\beta_{p^2}=
\frac{p^2+1-2p\jac{-1}{p}}{p^2(p^2-1)}, 
\end{equation}
where $\jac{\cdot}{p}$ denotes the Legendre symbol.
See also Lemma~\ref{lem:local-square-density}  and the equation~\eqref{eq:beta-m-def} below
for a general definition of $\beta_m$.

Since $\beta_{p^2}=p^{-2}+O(p^{-3})$, the following product converges
absolutely:
\begin{equation}
\label{eq:sq-singular-product-intro}
\mathfrak S_{\Rform}^{\mathrm{sq}}
=
\prod_p(1-\beta_{p^2}).
\end{equation}
Every factor in~\eqref{eq:sq-singular-product-intro} is positive, so absolute
convergence also gives $\mathfrak S_{\Rform}^{\mathrm{sq}}>0$.

\begin{theorem}
\label{thm:main-squarefree}
As $X\to\infty$,
\[
S_{\mathrm{sq}}(X)
=
\mathfrak S_{\Rform}^{\mathrm{sq}}N(X)
+O(X^{19/10+o(1)}).
\]
\end{theorem}

Next we estimate the frequency of almost prime values of $\Rform(\gamma)$.
For a positive integer $m$, let $\omega(m)$ denote the number of distinct prime divisors of $m$ and let $\Omega(m)$ denote the total number of prime factors of $m$, counted with multiplicity.

\begin{theorem}
\label{thm:almost-prime-lower-bound}
As $X\to\infty$, we have
\[
\#\{\gamma\in\Gamma_X:~\omega(\Rform(\gamma))= \Omega(\Rform(\gamma))\leq9\}
\gg \frac{N(X)}{\log X}.
\]
\end{theorem}

We note that Sedunova~\cite{Sed} gives a stronger result, with $9$ replaced by $7$. 
However, our previous comment about the flexibility of our approach still applies, so we believe 
 Theorem~\ref{thm:almost-prime-lower-bound} is indicative of what one can achieve by our 
 method in other related scenarios.

\subsection{General notation and conventions}
\label{sec:not}
We recall that  the notations $U = O(V)$, $U \ll V$ and $ V\gg U$  
are equivalent to $|U|\leqslant c V$ for some positive constant $c$, 
which throughout this work may depend on some parameters which we 
declare to be ``fixed''.   For example, in Lemma~\ref{lem:delta-exp} the constant 
may depend on the parameter $K$,  while in Lemma~\ref{lem:T approx} 
it may depend on $B$.

Furthermore,  $U \asymp V$  means that  $U \ll V \ll U$.

We also write $U = V^{o(1)}$ if for any fixed $\varepsilon>0$  we have 
$V^{-\varepsilon} \le |U |\le V^{\varepsilon}$ provided that $V$ is large 
enough.  

For a set $\cS$, we use $\one_{\cS}$ to denote the characteristic function of $\cS$
and, if $\cS$ is finite, we use  $\# \cS$ to denote its cardinality. 

 We use $\|\vy\|$ for  the Euclidean norm of $\vy \in \R^4$ and set
\[
\distZ{\vy}
=
\min_{\vn \in\Z^4}\|\vy-\vn\|.
\]
As usual, we set
\[
\e(t)=\exp(2\pi i t),
\qquad
\e_m(t)=\e(t/m).
\]

For an integer $m \ge 1$, we assume that the residue ring $\Z/m\Z$ is represented 
by the set $\{0, 1, \ldots, m-1\}$. In particular, we switch freely between elements 
of $\Z/m\Z$ and integers (and also apply the same convention to vectors).

We use $\tau$, $\mu$ and $\varphi$ to denote the divisor-counting function, the Möbius function and Euler's totient function, respectively.

We frequently use the classical bound 
\begin{equation}
\label{eq: tau}
\tau(q) = q^{o(1)};  
\end{equation}
see, for example,~\cite[Equation~(1.81)]{IwKow04}. 

The letter $p$ always denotes a prime number.

For integers $\nu \ge 1$ and $m \ne 0$, as usual,
$p^\nu \parallel m$ means $p^\nu \mid m$ and $p^{\nu+1} \nmid m$.

We often use the dot-product $\vb \cdot g$ between vectors 
$\vb = (b_1, \ldots, b_4) \in \R^4$ and $2\times 2$-matrices
\[
g = \begin{pmatrix}a_1&a_2\\ a_3&a_4\end{pmatrix},
\]
in which case we identify $g$ with a vector in $\R^4$  and thus 
 $\vb \cdot g = a_1b_1 + \ldots + a_4b_4$. 
 
We say that a function $\psi$ is \emph{smooth} if $\psi\in C^\infty$,
that is, if it has continuous derivatives of all orders.

Let $ \cD^\infty(\Omega)$ denote the set of smooth compactly supported functions on 
a certain domain $\Omega \subseteq \R^n$.

\subsection{Outline of the approach}  
We write $Q(\va)=a_1a_4-a_2a_3-1$ for $\va\in\Z^4$.
For $w \in C_c^\infty(\mathbb{R}^4)$,   define
\begin{equation}
\label{eq:A_dw}
A_{w,m}(X)
=
\sum_{\substack{\va\in\Z^4\\Q(\va)=0\\m\mid\Rform(\va)}}
 w(\va/X), 
\end{equation}
and
\begin{equation}
\label{eq:Nw}
N_w(X)=
\sum_{\substack{\va\in\Z^4\\Q(\va)=0}}
 w(\va/X).
\end{equation}

We first approximate $A_{w,m}(X)$ by $\beta_m N_w(X)$; 
see Lemma~\ref{lem:SmoothApprox} below, where $\beta_{m}$ is defined by~\eqref{eq:beta-m-def}. We use Heath-Brown's $\delta$-symbol method~\cite{HeathBrown96} in the form of Str\"ombergsson,  S\"odergren and  Vishe~\cite{StrSoVi26}; see 
Sections~\ref{sec; anal est},  \ref{sec:non-resonant} and~\ref{sec:resonant}. 
The resulting error bound in  Lemma~\ref{lem:SmoothApprox} depends explicitly on $m$ and the derivatives of $w$. 

We choose smooth product weights $w_{h,-}$ and $w_{h,+}$
 that bracket the characteristic function of $[-1,1]^4$. Then, applying the approximation of Lemma~\ref{lem:SmoothApprox}
 to $A_{w_{h,\pm},m}(X)$   yields an approximation for
\[
A_m(X) = \# \left\{ \gamma \in \Gamma_X :~m \mid \mathcal{R}(\gamma) \right\}
\]
by
$\beta_{m}N(X)$ uniformly for integers $1\leq m\leq X^{1/4}$; see 
Lemma~\ref{lem:circle-square-divisor-count}, where $N(X)$ is defined in~\eqref{eq:Gamma X}.

M\"obius inversion gives
$$
S_{\mathrm{sq}}(X)= \sum_{d\ge 1}
\mu(d)A_{d^2}(X).
$$
We then split the sum at $d = X^{1/10}$. The preceding estimate handles the smaller divisors. For the remaining terms, we apply the bound from Lemma~\ref{lem:fibre-bound} for the number of matrices  $\gamma\in\Gamma$ with a prescribed value of $\Rform(\gamma)$, which is 
obtained by following an argument of Friedlander and Iwaniec~\cite{FrIw09}.
This controls the tail, while the multiplicativity of $\beta_m$ leads to the Euler product~\eqref{eq:sq-singular-product-intro} in the main-term. 

\section{Local densities}

\subsection{Densities modulo prime powers}

We interpret $\beta_{p^2}$, given by the equation~\eqref{eq:beta-p},  as the density of
$g\in\cG_{p^2}$ with $\Rform(g)\equiv0\bmod {p^2}$, where,
for an integer $s\ge1$, we write
\begin{equation}
\label{eq:def_G_m}
    \cG_s=\SL_2(\Z/s\Z).
\end{equation}
Here we investigate similar densities modulo arbitrary positive integers, which may also be 
useful for other applications.

A simple argument, based on the Chinese Remainder Theorem and Hensel lifting, 
yields
\begin{equation}
\label{eq:order_Gm}
\# \cG_s = s^3\prod_{p\mid s}(1-p^{-2});
\end{equation}
see also~\cite[Exercise~1.2.3(b)]{DiSh}.

\begin{lem} 
\label{lem:local-square-density}
For every prime $p$ and every integer $\nu \geq 1$, we have
\[
  \frac{
  \#\left\{g \in \cG_{p^\nu} :
  \Rform(g) \equiv 0 \bmod{p^\nu}\right\}
  }{
  \#\cG_{p^\nu}
  }
  =
  \beta_{p^\nu},
\]
where
\begin{equation*}
    \beta_{p^\nu}
  =
  \begin{cases}
  \displaystyle
  \frac{\left(p-\jac{-1}{p}\right)^2}
       {p^\nu(p^2-1)},
  & p \text{ odd},\\[10pt]
  \displaystyle \frac{1}{3},
  & p=2,\ \nu=1,\\[6pt]
  0,
  & p=2,\ \nu \geq 2.
  \end{cases}
\end{equation*}
\end{lem}

\begin{proof}
We first consider $p=2$. Of the six elements of $\SL_2(\mathbb F_2)$, two have two nonzero entries and the remaining four have three. Since $x^2=x$ in $\mathbb F_2$, it follows that
\[\beta_2=\frac{2}{6}=\frac13.\]
If $\nu\geq2$, reduction modulo $2$ shows that every $g\in\cG_{2^\nu}$ has either two or three odd entries. Since an integer square is congruent to $0$ or $1$ modulo $4$, according as the integer is even or odd, we have
\[\Rform(g)\equiv2\ \text{or}\ 3\bmod4.\]
Thus $2^\nu\nmid\Rform(g)$ and hence $\beta_{2^\nu}=0$.

Let $p$ be odd and let $T_{\pm,\nu}$ denote the number of solutions to
\begin{equation}
\label{eq:mod-pnu}
X^2+Y^2\equiv\pm2\bmod{p^\nu},
\qquad X,Y\in\mathbb Z/p^\nu\mathbb Z.
\end{equation}
For
\[
g=
\begin{pmatrix}
a_1&a_2\\
a_3&a_4
\end{pmatrix},
\]
put
\[
X=a_1+a_4,\qquad Y=a_2-a_3,\qquad
U=a_1-a_4,\qquad V=a_2+a_3.
\]
Then
\[
X^2+Y^2=\Rform(g)+2\det g,
\qquad
U^2+V^2=\Rform(g)-2\det g.
\]
Since $2$ is invertible modulo $p^\nu$, this change of variables is
bijective. Hence the conditions
\[
\det g\equiv1\bmod{p^\nu},
\qquad
\Rform(g)\equiv0\bmod{p^\nu}
\]
are equivalent to
\[
X^2+Y^2\equiv2\bmod{p^\nu},
\qquad
U^2+V^2\equiv-2\bmod{p^\nu}.
\]
Consequently,
\[
\#\left\{
g\in\cG_{p^\nu}:
\Rform(g)\equiv0\bmod{p^\nu}
\right\}
=
T_{+,\nu}T_{-,\nu}.
\]

By~\cite[Lemma~6.24]{LN}, each of the congruences
\[
x^2+y^2\equiv\pm2\bmod p
\]
has
\[
p-\jac{-1}{p}
\]
solutions modulo $p$.

We next count their lifts. Fix $1\leq j<\nu$ and a solution
$(x,y)$ modulo $p^j$. Choosing integer representatives, write
\[
x^2+y^2=\pm2+p^jz
\]
for some $z\in\mathbb Z$. Every lift modulo $p^{j+1}$ is of the form
\[
X=x+up^j,\qquad Y=y+vp^j,
\qquad u,v\in\mathbb Z/p\mathbb Z.
\]
Substitution into~\eqref{eq:mod-pnu}, with $\nu$ replaced by $j+1$ shows that such a lift is a
solution if and only if
\[
2(xu+yv)\equiv-z\bmod p.
\]
Since $(x,y)\not\equiv(0,0)\bmod p$, this is a nontrivial linear
congruence in $u$ and $v$, and therefore has exactly $p$
solutions. Iterating the lifting procedure gives
\[
T_{\pm,\nu}
=
p^{\nu-1}\left(p-\jac{-1}{p}\right).
\]
It follows that
\[
\#\left\{
g\in\cG_{p^\nu}:
\Rform(g)\equiv0\bmod{p^\nu}
\right\}
=
p^{2\nu-2}
\left(p-\jac{-1}{p}\right)^2.
\]
Finally, by~\eqref{eq:order_Gm},
\[
\#\cG_{p^\nu}=p^{3\nu-2}(p^2-1).
\]
Therefore
\[
\frac{
\#\left\{
g\in\cG_{p^\nu}:
\Rform(g)\equiv0\bmod{p^\nu}
\right\}
}{
\#\cG_{p^\nu}
}
=
\frac{
\left(p-\jac{-1}{p}\right)^2
}{
p^\nu(p^2-1)
}.
\]
\end{proof}

In particular, when $\nu=2$, Lemma~\ref{lem:local-square-density} agrees with the definition of $\beta_{p^2}$ in~\eqref{eq:beta-p}.

\subsection{Densities modulo $m$ and the inclusion-exclusion principle}

We first record the corresponding local density for arbitrary integer moduli.

For every integer $m \geq 1$, define
\begin{equation}
\label{eq:beta-m-def}
    \beta_m = \prod_{p^\nu \parallel m} \beta_{p^\nu},
\end{equation}
where the empty product, corresponding to $m=1$, is understood to be~$1$.
The Chinese Remainder Theorem, together with Lemma~\ref{lem:local-square-density}, now yields the following density formula for every modulus $m$.

\begin{cor}
\label{cor:beta_d2}
For every integer $m \ge 1$,
\[
\frac{
\#\left\{g\in\SL_2(\Z/m\Z):~\Rform(g)\equiv0\bmod {m}\right\}
}{
\#\SL_2(\Z/m\Z)
}
= \beta_{m}.
\]
\end{cor}

\begin{lemma}
\label{lem:beta-d-estimates}
As $m\to\infty$, we have
\[
    \beta_m \leq m^{-1+o(1)}.
\]
\end{lemma}

\begin{proof}
Lemma~\ref{lem:local-square-density} gives
\[
    0 \leq \beta_{p^\nu} \leq 2p^{-\nu}
\]
for every prime $p$ and every integer $\nu \geq 1$. Therefore, using the definition~\eqref{eq:beta-m-def}, we obtain
\[
    \beta_m = \prod_{p^\nu \parallel m} \beta_{p^\nu} \leq 2^{\omega(m)} \prod_{p^\nu \parallel m} p^{-\nu} = 2^{\omega(m)}m^{-1},
\]
where $\omega(m)$ denotes the number of distinct prime divisors of $m$. Since
\[
    2^{\omega(m)} \leq \tau(m) = m^{o(1)}
\]
by~\eqref{eq: tau}, the result follows.
\end{proof}

In particular, we see from Lemma~\ref{lem:beta-d-estimates} that  
\[
\sum_{\substack{d>D\\d\ \mathrm{squarefree}}}\beta_{d^2} \le D^{-1+o(1)},
\]
as $D\to \infty$.

The inclusion-exclusion principle implies the expansion
\begin{equation}
\label{eq:incl/excl}
S_{\mathrm{sq}}(X)= \sum_{d\ge 1}
\mu(d)A_{d^2}(X),
\end{equation}
with the M\"obius function $\mu(d)$. Clearly, the sum is finite and terminates for $d \ge 2X$.

We next obtain an asymptotic formula for $A_{d^2}(X)$ when $d$ is small and an upper bound for the remaining values of $d$.

\section{Background on the $\delta$-method}
\label{sec; anal est}

\subsection{Some norm estimates}

For $w\in\cD^\infty(\R^4)$, we write $\partial_i=\partial/\partial u_i$ with the convention that $\partial^0_i w = w$.

For functions defined on a Euclidean space, $\|\cdot\|_1$ and $\|\cdot\|_\infty$ denote the standard $L^1$- and $L^\infty$-norms.

For $0<h\leq1$, we introduce the
{\it edge\/} norm
\begin{equation}
\label{eq:edge-norm}
\begin{split}
\|w\|_{\cE_h^5}
&=
\|w\|_1+\|w\|_\infty\\
&\quad+
\sum_{i=1}^4\sum_{j=1}^5
\(
 h^j\|\partial_i^jw\|_\infty
 +h^{j-1}\|\partial_i^jw\|_1
\).
\end{split}
\end{equation}

The name ``edge'' of the norm~\eqref{eq:edge-norm} is motivated by one-dimensional cutoff functions that vary only within the $h$-neighborhoods of an interval's endpoints and by their tensor products in $\mathbb{R}^4$, which vary only near the boundary of a box.
We note that the definition~\eqref{eq:edge-norm}
involves no mixed partial derivatives with respect to different variables.

As usual, $\supp w$ denotes the support of a function $w$.

\begin{lemma}
\label{lem:adjustable-edge-cutoffs}
For every $0<h\leq1/2$, there are functions
$\psi_h^-,\psi_h^+\in \cD^\infty(\R)$ with 
\[
0\leq\psi_h^-(x) \leq\one_{[-1,1]}(x)\leq\psi_h^+(x) \leq1,
\]
and
\[
\psi_h^-(x)=1\quad (|x|\leq1-h),
\qquad
\supp \psi_h^+\subseteq[-1-h,1+h],
\]
and such that 
\begin{itemize}
\item[(i)]  for every fixed integer $j\geq1$,
\[
\| (\psi_h^\pm)^{(j)}\|_\infty\ll h^{-j} \mand 
\| (\psi_h^\pm)^{(j)}\|_1\ll h^{1-j};
\]

\item[(ii)] for
\[
w_{h,\pm}(\vu)=\prod_{i=1}^4\psi_h^\pm(u_i),
\]
we have 
\[
\|w_{h,\pm}\|_{\cE_h^5}\ll1. 
\]
\end{itemize}
\end{lemma}

\begin{proof}
Choose a smooth function  $\vartheta\in C^\infty(\R)$ with
$0\leq\vartheta(t) \leq1$, $\vartheta(t)=0$ for $t\leq0$ and
$\vartheta(t)=1$ for $t\geq1$, and set
\begin{align*}
& \psi_h^-(x)
=
\vartheta\(\frac{x+1}{h}\)
\vartheta\(\frac{1-x}{h}\),\\
&
\psi_h^+(x)
=
\vartheta\(\frac{x+1+h}{h}\)
\vartheta\(\frac{1+h-x}{h}\).
\end{align*}
Lower and upper bounds as well as the support and flatness assertions are immediate.  

Next, every nonzero derivative
of either function is supported in the union of two intervals of total
length $O(h)$.  The chain rule therefore gives~(i). 

For the product weights, we have
\[
\|w_{h,\pm}\|_\infty\leq1,
\qquad
\|w_{h,\pm}\|_1\ll1.
\]
Moreover, for $1\leq i\leq4$ and $1\leq j\leq5$,
\[
\partial_i^jw_{h,\pm}(\vu)
=
(\psi_h^\pm)^{(j)}(u_i)
\prod_{\substack{k=1\\k\ne i}}^4\psi_h^\pm(u_k).
\]
Since $\|\psi_h^\pm\|_1\ll1$ and
$\|\psi_h^\pm\|_\infty\leq1$, Part~(i) gives
\[
\|\partial_i^jw_{h,\pm}\|_\infty\ll h^{-j},
\qquad
\|\partial_i^jw_{h,\pm}\|_1\ll h^{1-j}.
\]
Substitution into~\eqref{eq:edge-norm} proves~(ii).
\end{proof}

\subsection{Approximation to the $\delta$-symbol}

For $n \in \Z$,  let $\delta_0 = \vec{1}_{\{0\}}$, that is, 
\[
\delta_0(n)=
\begin{cases}
1,& \text{if} \ n=0,\\
0,&\text{if} \  n\ne0.
\end{cases}
\]

We use  $\Sigma_{c\bmod q}^*$ to denote summation over the reduced residues modulo $q$. 

We need the following representation of the
$\delta$-symbol, which is related to pioneering work of 
Heath-Brown~\cite[Theorem~1]{HeathBrown96}.

\begin{lemma} 
\label{lem:delta-exp} There exists a family of smooth functions
$p_{q, X}:\R\to\C$, indexed by $X\geq1$ and positive integers $q\leq X$,
with the following properties. For every fixed $K\geq1$, we have
\[
p_{q, X}(z)\ll (1+qX|z|)^{-K}, 
\] 
and 
\[
p_{q, X}(z)
=
1+O\(
(1+X^2|z|)^{2K+2}\(\frac qX\)^K
\),
\]
and, for every integer $n$,
\[
\delta_0(n)
=
\sum_{1\leq q\leq X}\  \sideset{}{^*}\sum_{c\bmod q}
\int_{\R}p_{q, X}(z)\e\(\(\frac cq+z\)n\)\,dz
+O(X^{-K}). 
\] 
\end{lemma}

\begin{proof}
The expansion of $\delta_0(n)$ and the decay estimate on $p_{q, X}(z)$ are 
stated in~\cite[Proposition~5.3]{StrSoVi26}, while 
the asymptotic formula for $p_{q, X}(z)$ is given in~\cite[Lemma~2.2]{MarmonVishe19}. 
\end{proof}

We emphasise that the functions $p_{q, X}:\R\to\C$ in Lemma~\ref{lem:delta-exp}
 are independent of $K$ and the implied constants depend only on $K$.

\subsection{Delta expansion and oscillatory estimates}

Following~\cite[Section~5]{StrSoVi26}, we define the determinant
polynomial and its homogeneous part  
\begin{equation}
\label{eq:Q-Q0-def}
Q(\va)=a_1a_4-a_2a_3-1,
\qquad
Q_0(\va)=a_1a_4-a_2a_3,
\end{equation}
where $\va=(a_1,a_2,a_3,a_4)\in\R^4$.

For $X\geq1$, $w\in\cD^\infty(\R^4)$ and
$\balpha\in\R^4$, define
\begin{equation}
\label{eq:TwX}
T_w(X;\balpha)
=
\sum_{\substack{\va\in\Z^4\\Q(\va)=0}}
\e(\balpha\cdot\va)w(\va/X).
\end{equation}

When $\balpha\in\Q^4$, we denote by $\fs(\balpha)$ the least positive integer $s$
such that $s\balpha\in\Z^4$.   We recall the definition of $\cG_s$ as in~\eqref{eq:def_G_m}
and   set
\begin{equation}
\label{eq:lambda-s-b}
\lambda(\balpha)
=
\frac{1}{\# \cG_s}\sum_{g\in \cG_s}\e_s(\vb\cdot g),
\end{equation}
where $s = \fs (\balpha)$,  $\vb=s\balpha$ and, as we have mentioned in Section~\ref{sec:not}, in the dot-products $\vb\cdot g$,  the matrix $g$ is identified with the vector of its four entries.

Put
\begin{equation}
\label{eq:Q1-def}
Q_1(\vu)=Q_0(\vu)-X^{-2},
\end{equation}
so that $Q(X\vu)=X^2Q_1(\vu)$ and, for $z \in \R$ and $\vv\in \R^4$, define
\begin{equation}
\label{eq:I-z-v-def}
I_w(z,\vv)
=
\int_{\R^4}w(\vu)
\e\(zQ_1(\vu)+\vv\cdot\vu\)
\,d\vu,
\end{equation}
and, for an integer $q\geq1$, we also define
\begin{equation}
\label{eq:S-q-v-def}
S(q,\vv)
=
 \sideset{}{^*}\sum_{c\bmod q}\ \sum_{\vx\bmod q}
\e_q\(cQ(\vx)+\vv\cdot\vx\).
\end{equation}

\begin{lemma} 
\label{lem:T approx} 
Fix some $B\geq1$. Let
$w\in\cD^\infty(\R^4)$ satisfy
$\supp w\subseteq[-B,B]^4$. Then,  for
$\balpha\in\R^4$,
\begin{equation}
\label{eq:delta-poisson-formula}
\begin{split}
T_w(X;\balpha)
&=
X^4\sum_{1\leq q\leq X}\frac1{q^4}
\sum_{\vv\in\Z^4}S(q,\vv)\\
&\quad\times
\int_{\R}p_{q, X}(z)
I_w\(zX^2,\frac Xq(q\balpha-\vv)\)\,dz\\
&\qquad \qquad \qquad \qquad \qquad \qquad \qquad +O(\|w\|_\infty).
\end{split}
\end{equation}
\end{lemma}

\begin{proof}
The calculation in~\cite[Equations~(59)--(62)]{StrSoVi26} applies directly to the
sum $T_w(X;\balpha)$.  With the notation~\eqref{eq:Q1-def}, \eqref{eq:I-z-v-def} and~\eqref{eq:S-q-v-def},   
the resulting formula is exactly~\eqref{eq:delta-poisson-formula}.
\end{proof}

We also need the following analogue of~\cite[Equation~(64)]{StrSoVi26}. 

\begin{lemma} 
\label{lem:complete-exponential-sum}
For every $q\geq1$ and $\vv\in\Z^4$, 
\[
|S(q,\vv)|\leq\tau(q)q^{5/2}.
\]
\end{lemma}

\begin{proof}
A specialisation of~\cite[Equation~(63)]{StrSoVi26} to 
the determinant polynomial~\eqref{eq:Q-Q0-def} 
(after we follow the same transformations as in~\cite{StrSoVi26}) implies
\[
S(q,\vv)
=
q^2 \sideset{}{^*}\sum_{c\bmod q}
\e_q\(-c-c^{-1}Q_0(\vv)\).
\]
Applying the Weil  bound for Kloosterman sums in the form~\cite[Corollary~11.12]{IwKow04} 
gives the desired result.
\end{proof}

Note that for the symmetric matrix  
\[
M_0=
\begin{pmatrix}
0&0&0&1\\
0&0&-1&0\\
0&-1&0&0\\
1&0&0&0
\end{pmatrix},
\]
we have 
\[
\left(
\frac{\partial Q_0(\vu)}{\partial u_1},\ldots,
\frac{\partial Q_0(\vu)}{\partial u_4}
\right)
=\vu M_0,
\]
where $\vu = (u_1, \ldots,  u_4)$ 
is the row vector of variables.

Next we  record the standard $C^5$ estimate from~\cite[Equation~(67)]{StrSoVi26} together with an edge-sensitive refinement in which
the differentiated weight remains inside the integral.

\begin{lemma} 
\label{lem:oscillatory-integral}
Let $B\geq1$ be fixed.  Assume that   $w\in \cD^\infty(\R^4)$ satisfies
$\supp w\subseteq [-B,B]^4$.  For $z\in\R$, put
\[
\delta=(1+|z|)^{-1/2},
\]
and define
\[
K_{z,\vv}(\vu)
=
\(1+\delta\|z\vu M_0+\vv\|\)^{-5}.
\] 
Uniformly for $z\in\R$ and $\vv\in\R^4$, we have 
\[
I_w(z,\vv)
\ll 
\sum_{i=1}^4\sum_{j=0}^5
\delta^j
\int_{\R^4}|\partial_i^jw(\vu)|
K_{z,\vv}(\vu)\,d\vu.
\]
\end{lemma}

\begin{proof}
Choose a non-negative
function $w_0$ such that
\[
w_0\in \cD^\infty((0,1)^4),
\qquad
\int_{\R^4}w_0(\vy)\,d\vy=1.
\]
For each fixed $\vy$, translate the variable of integration in~\eqref{eq:I-z-v-def} by $\delta \vy$.  Averaging the resulting identity
against $w_0(\vy)$ and using
\[
Q_1(\vu+\delta\vy)
=
Q_1(\vu)
+\delta(\vu M_0)\cdot\vy
+\delta^2Q_0(\vy),
\]
we obtain
\[
I_w(z,\vv)
=
\int_{\R^4}
\e\(zQ_1(\vu)+\vv\cdot\vu\)
J_w(\vu)\,d\vu,
\]
where
\[
J_w(\vu)
=
\int_{[0,1]^4}
 w_0(\vy)w(\vu+\delta\vy)
 \e\(z\delta^2Q_0(\vy)
       +\delta(z\vu M_0+\vv)\cdot\vy\)
\,d\vy.
\]

Set
\[
\bell(\vu)=z\vu M_0+\vv = \(\ell_1(\vu), \ldots, \ell_4(\vu)\).
\]
Partition $\R^4$ into measurable sets $\Omega_1,\ldots,\Omega_4$ so that,
for $\vu\in\Omega_i$, the $i$th coordinate of
$\bell(\vu)$ has maximal absolute value; ties are resolved
by choosing the smallest index.  Thus, for $\vu\in\Omega_i$,
\begin{equation}
\label{eq:maximal-coordinate-comparison}
|\ell_i(\vu)| \leq \|\bell(\vu)\| \leq 2|\ell_i(\vu)|.
\end{equation} 
Fix $\vu\in\Omega_i$ and write
\[
A_{i,\vu}(\vy) = w_0(\vy) 
\e\( z\delta^2Q_0(\vy) +\delta\sum_{k\neq i}\ell_k(\vu)y_k\).
\]
Then
\[
J_w(\vu) =\int_{[0,1]^4} A_{i,\vu}(\vy)
 w(\vu+\delta\vy) \e\(\delta\ell_i(\vu)y_i\)\,d\vy.
\]
The second term in the phase defining $A_{i,\vu}$ is independent of
$y_i$.  Moreover, $Q_0$ is linear in each coordinate and
$|z|\delta^2\leq1$.  Since $w_0$ is fixed, it follows that
\begin{equation}
\label{eq:Ai-derivative-bound}
\sup_{\vy\in\R^4}
|\partial_{y_i}^rA_{i,\vu}(\vy)|
\ll 1,
\qquad 0\leq r\leq5,
\end{equation}
uniformly in $z$, $\vv$ and $\vu\in\Omega_i$.

Put $\rho=\delta|\ell_i(\vu)|$.  If $\rho\geq1$, then five
integrations by parts in $y_i$ give
\[
J_w(\vu)
=
\frac{(-1)^5}{(2\pi i\delta\ell_i(\vu))^5}
\int_{[0,1]^4}
\partial_{y_i}^5
\(A_{i,\vu}(\vy)
      w(\vu+\delta\vy)\)
\e\(\delta\ell_i(\vu)y_i\)
\,d\vy.
\]
There are no boundary terms because $A_{i,\vu}$ is compactly supported
in $(0,1)^4$.  By the chain rule,
\[
\partial_{y_i}^j w(\vu+\delta\vy)
=
\delta^j(\partial_i^jw)(\vu+\delta\vy),
\qquad 0\leq j\leq5.
\]
The Leibniz rule and~\eqref{eq:Ai-derivative-bound} therefore yield
\[
J_w(\vu)
\ll
\rho^{-5}
\sum_{j=0}^5\delta^j
\int_{[0,1]^4}
|\partial_i^jw(\vu+\delta\vy)|\,d\vy.
\]
If $\rho<1$, the same estimate without the factor $\rho^{-5}$ follows by
taking absolute values, using the term $j=0$ on the right.  In the ranges
$\rho<1$ and $\rho\geq1$, respectively, the factors $1$ and
$\rho^{-5}$ are both $O((1+\rho)^{-5})$.  Moreover,~\eqref{eq:maximal-coordinate-comparison} implies
$1+\rho\gg1+\delta\|\bell(\vu)\|$.  Thus, recalling the definition of $K_{z,\vv}(\vu)$,  we obtain
\begin{equation}
\label{eq:J-pointwise-bound}
J_w(\vu)
\ll
K_{z,\vv}(\vu)
\sum_{j=0}^5\delta^j
\int_{[0,1]^4}
|\partial_i^jw(\vu+\delta\vy)|\,d\vy
\qquad (\vu\in\Omega_i).
\end{equation}

Taking absolute values in the outer integral and applying~\eqref{eq:J-pointwise-bound}, we obtain
\[
I_w(z,\vv)
\ll
\sum_{i=1}^4\sum_{j=0}^5\delta^j
\int_{\Omega_i}\int_{[0,1]^4}
K_{z,\vv}(\vu)
|\partial_i^jw(\vu+\delta\vy)|
\,d\vy\,d\vu.
\]
For fixed $\vy$, set $\vx=\vu+\delta\vy$.  Then
\[
\delta\bell(\vu)
=
\delta\bell(\vx)
-z\delta^2\vy M_0.
\]
Set
\[
C_0=\sup_{\vy\in[0,1]^4}\|\vy M_0\|<\infty.
\]
Since $|z|\delta^2\leq1$, we have
\[
\|z\delta^2\vy M_0\|\leq C_0
\qquad (\vy\in[0,1]^4).
\]
The triangle inequality consequently gives
\begin{equation}
\label{eq:kernel-bounded-translation}
K_{z,\vv}(\vx-\delta\vy)
\asymp
K_{z,\vv}(\vx),
\qquad \vy\in[0,1]^4,
\end{equation}
with comparison constants depending only on $M_0$.  Tonelli's theorem~\cite[Theorem~2.37]{Folland99},
the bound~\eqref{eq:kernel-bounded-translation}, the inclusion
$\Omega_i+\delta\vy\subset\R^4$ and
$\vol([0,1]^4)=1$ now give
\[
\begin{aligned}
I_w(z,\vv)
&\ll
\sum_{i=1}^4\sum_{j=0}^5\delta^j
\int_{[0,1]^4}
\int_{\Omega_i+\delta\vy}
K_{z,\vv}(\vx)
|\partial_i^jw(\vx)|
\,d\vx\,d\vy \\
&\leq
\sum_{i=1}^4\sum_{j=0}^5\delta^j
\int_{\R^4}
K_{z,\vv}(\vx)
|\partial_i^jw(\vx)|\,d\vx.
\end{aligned}
\]
This concludes the proof. 
\end{proof}

We now record two useful consequences of
Lemma~\ref{lem:oscillatory-integral}, formulated in terms of
\begin{equation}
\label{eq:Whw}
 \cW_{w,h} = h^{-4}\|w\|_{\cE_h^5},
\end{equation}
where $\|\cdot\|_{\cE_h^5}$ is the edge norm defined
in~\eqref{eq:edge-norm}. 
We shall use the notation~\eqref{eq:Whw} throughout
the remainder of the paper.

\begin{cor} 
\label{cor:oscillatory-integral} Under the conditions of Lemma~\ref{lem:oscillatory-integral}, 
we have:
\begin{itemize}
\item[(i)]  for $0<h\leq1$, if for some  $L\geq0$ we have
\begin{equation}
\label{eq:phase-gradient-lower-bound}
\|z\vu M_0+\vv\|\geq L
\qquad (\vu\in\supp w),
\end{equation}
then
\[
I_w(z,\vv)
\ll
 \cW_{w,h} (1+\delta L)^{-5};
\]
\item[(ii)]   for every $0<h\leq1$, uniformly in $z$ and $\vv$, one also has
\[
I_w(z,\vv)
\ll
 \cW_{w,h} (1+|z|)^{-2}.
\]
\end{itemize}
\end{cor}

\begin{proof}
To prove (i), suppose that~\eqref{eq:phase-gradient-lower-bound} holds.  Since every
$\partial_i^jw$ is supported in $\supp w$, Lemma~\ref{lem:oscillatory-integral} gives
\[
I_w(z,\vv)
\ll
(1+\delta L)^{-5}
\(
\|w\|_1
+
\sum_{i=1}^4\sum_{j=1}^5
\delta^j\|\partial_i^jw\|_1
\).
\]
By~\eqref{eq:edge-norm}, we have
$\|\partial_i^jw\|_1\leq h^{1-j}\|w\|_{\cE_h^5}$ for
$1\leq j\leq5$.  Since $0<h,\delta\leq1$ and
$h^{1-j}\leq h^{-4}$ throughout this range, we obtain~(i).

We now prove~(ii).  If $|z|\leq1$, then Part~(i), applied with
$L=0$, gives
\[
I_w(z,\vv)\ll\cW_{w,h}.
\]
Since $(1+|z|)^{-2}\asymp1$ in this range, the desired estimate
follows.
Suppose now that $|z|>1$.  Since $|\det M_0|=1$, the affine change of
variables
\[
\vt=\delta(z\vu M_0+\vv)
\]
gives
\begin{equation}
\label{eq:kernel-four-dimensional-volume}
\int_{\R^4}K_{z,\vv}(\vu)\,d\vu
=
\delta^{-4}|z|^{-4}
\int_{\R^4}(1+\|\vt\|)^{-5}\,d\vt
\ll
(1+|z|)^{-2}.
\end{equation}
For $0\leq j\leq4$, the definition of the edge norm gives
\[
\|\partial_i^jw\|_\infty
\leq
h^{-j}\|w\|_{\cE_h^5},
\]
where, due to our previous convention,  $\partial_i^0w=w$.  Since $\delta\leq1$ and $0<h\leq1$,
we have $\delta^jh^{-j}\leq h^{-4}$ for $0\leq j\leq4$.  Hence, by~\eqref{eq:kernel-four-dimensional-volume},
\[
\delta^j
\int_{\R^4}|\partial_i^jw(\vu)|
K_{z,\vv}(\vu)\,d\vu
\ll
h^{-4}\|w\|_{\cE_h^5}(1+|z|)^{-2}
\qquad (0\leq j\leq4).
\]
For $j=5$, we instead use $K_{z,\vv}\leq1$ and, recalling the choice of $\delta$ in Lemma~\ref{lem:oscillatory-integral}, we obtain 
\[
\delta^5
\int_{\R^4}|\partial_i^5w(\vu)|
K_{z,\vv}(\vu)\,d\vu
\leq
\delta^5\|\partial_i^5w\|_1
\leq
h^{-4}\|w\|_{\cE_h^5}(1+|z|)^{-5/2}.
\]
Summing over $i$ and $j$ proves~(ii).
\end{proof}

\section{Estimating the non-resonant contribution}
\label{sec:non-resonant}

\subsection{Preliminary discussion}
In the Poisson expansion~\eqref{eq:delta-poisson-formula}, the linear frequency
in the oscillatory integral~\eqref{eq:I-z-v-def} is
$X(q\balpha-\vv)/q$.  It vanishes precisely when
$\vv=q\balpha\in\Z^4$, in which case we call such a vector 
$\vv$ resonant. 

 We separate this resonant vector, when it exists,
from the remaining dual frequencies.  The denominator of a rational twist 
$\balpha$ 
therefore determines the moduli at which resonance occurs.  

For a positive integer $q\leq X$,   a real vector $\balpha\in\R^4$ and a function 
$w\in \cD^\infty(\R^4)$, we introduce 
 the quantity
\[
\cT_{w,q}^\circ(\balpha) =
\frac{X^4}{q^4} \sideset{}{^\circ}\sum_{\vv\in\Z^4}S(q,\vv)
\int_{\R}p_{q, X}(z)
I_w\(zX^2,\frac Xq(q\balpha-\vv)\)\,dz,
\]
where we use  $\Sigma_{\vv\in\Z^4}^\circ$ to denote  
that the term corresponding to $\vv=q\balpha$ is omitted from the summation (when $q\balpha\in\Z^4$). 
 
Thus $\cT_{w,q}^\circ(\balpha)$ is precisely the contribution of the
non-resonant dual frequencies to the term corresponding to the modulus $q$
in~\eqref{eq:delta-poisson-formula}.

\subsection{Bounding the non-resonant contribution}

We now estimate non-resonant contributions $\cT_{w,q}^\circ(\balpha)$ in terms of the norm $\|w\|_{\cE_h^5}$ (for an arbitrary $h\in (0,1]$).  We recall the definition of $ \cW_{w,h}$ in~\eqref{eq:Whw}.

\begin{lemma} 
\label{lem:nonresonant-dual-frequencies} 
Let $B\geq1$ be fixed. Let $X\geq1$ and let $q$ be an integer
satisfying $1\leq q\leq X$. Suppose that $\balpha\in\R^4$ and that
$w\in\cD^\infty(\R^4)$ satisfies
\[
\supp w\subseteq[-B,B]^4.
\]   Then, for every $0<h\leq1$,
\[
\cT_{w,q}^\circ(\balpha) 
\ll \cW_{w,h}
X^2\frac{\tau(q)}{q^{3/2}}
\begin{cases}
\displaystyle
\(1+\frac Xq\distZ{q\balpha}\)^{-1},
&  \text{if} \ q\balpha\notin\Z^4,\\[8pt]
\displaystyle
\frac qX,
&  \text{if} \  q\balpha\in\Z^4.
\end{cases}
\]
\end{lemma}

\begin{proof}
Fix an integer $K\geq5$.  Using the decay estimate on $p_{q, X}(z)$ from Lemma~\ref{lem:delta-exp} and the bound on $S(q,\vv)$ from Lemma~\ref{lem:complete-exponential-sum} and substituting $t=zX^2$,  we derive 
\begin{equation}
\label{eq:Tq-to-I-sum}
\cT_{w,q}^\circ(\balpha)
\ll 
X^2\frac{\tau(q)}{q^{3/2}}\cK_q^\circ(\balpha),
\end{equation}
where
\[
\cK_q^\circ(\balpha)
=
\int_{\R}
\(1+\frac{q|t|}{X}\)^{-K}
\sideset{}{^\circ}\sum_{\vv\in\Z^4}
\left|I_w\(t,\frac Xq(q\balpha-\vv)\)\right|\,dt.
\]
It therefore suffices to prove
\begin{equation}
\label{eq:Kq-alpha-bound}
\cK_q^\circ(\balpha)
\ll
\cW_{w,h}
\begin{cases}
\displaystyle
\(1+\dfrac Xq\distZ{q\balpha}\)^{-1},
&  \text{if} \  q\balpha\notin\Z^4,\\[8pt]
\displaystyle
q/X,
&  \text{if} \  q\balpha\in\Z^4,
\end{cases}
\end{equation}
(with the implied constant depending on $B$ and $K$).

For $t\in\R$ and $\bxi\in\R^4$, the phase of
$I_w(t,\bxi)$ has the gradient $t\vu M_0+\bxi$. 

Assume that for some $L\geq0$ we have
\[
\|t\vu M_0+\bxi\|\geq L,
\qquad \vu\in\supp w.
\]
Then, by Corollary~\ref{cor:oscillatory-integral}~(i), we have
\begin{equation}
\label{eq:I-two-cost-frequency}
I_w(t,\bxi)
\ll 
\cW_{w,h}  \(1+(1+|t|)^{-1/2}L\)^{-5}.
\end{equation}

Without any condition on the gradient $t\vu M_0+\bxi$, 
by Corollary~\ref{cor:oscillatory-integral}~(ii) 
one has
\begin{equation}
\label{eq:I-two-cost-stationary}
I_w(t,\bxi) \ll \cW_{w,h} (1+|t|)^{-2}.
\end{equation}

For brevity, put
\begin{equation}
\label{eq:xiv}
\bxi_{\vv} = \frac Xq(q\balpha-\vv),  \qquad \vv\in\Z^4.
\end{equation}

In what follows, we repeatedly use the elementary estimates
\begin{equation}
\label{eq: lattice estim}
\#\{\vv\in\Z^4:\|\vv-\boldsymbol\vartheta\|\leq R\} \ll(1+R)^4, 
\end{equation}
and 
\begin{equation}
\label{eq: lattice sum}
\sum_{\substack{\vv\in\Z^4\\ \|\vv-\boldsymbol\vartheta\|\geq1/2}} \|\vv-\boldsymbol\vartheta\|^{-5} \ll 1,
\end{equation}
uniformly for $\boldsymbol\vartheta\in\R^4$ and $R\geq0$.

Set
\[
B_1=B+1 \mand c_0=(8B_1)^{-1}.
\] 
We first consider the case 
\begin{equation}
\label{eq:outer-t-range}
|t|\geq  c_0X/q.
\end{equation}

Following the idea of~\cite[Section~5.4]{StrSoVi26}, we call $\vv\in\Z^4$ {\it good\/} if
\[
\|\bxi_{\vv}\| = \frac Xq\|q\balpha-\vv\|\geq4B_1|t|
\]
and {\it bad\/} otherwise. 

Since $M_0$ is a signed permutation matrix,
$\|\vu M_0\|=\|\vu\|\leq2B$ on $\supp w$.  Hence, for a
good vector $\vv$ and any $\vu\in\supp w$,
\[
\|t\vu M_0+\bxi_{\vv}\|
\geq\frac12\|\bxi_{\vv}\|
=
\frac{X}{2q}\|q\balpha-\vv\|,
\]
where we recall that $\bxi_{\vv}$ is given by~\eqref{eq:xiv}.
Moreover,~\eqref{eq:outer-t-range} implies
$\|q\balpha-\vv\|\geq1/2$ for every good vector.  Therefore~\eqref{eq:I-two-cost-frequency}, 
taken with
\[
L=\frac{X}{2q}\|q\boldsymbol\alpha-\vv\|,
\]
and~\eqref{eq: lattice sum} give
\begin{equation}
\label{eq:outer-good-vectors-edge}
\begin{aligned}
\sum_{\substack{\vv\in\Z^4\\ \vv\ \mathrm{good}}}
|I_w(t,\bxi_{\vv})|
&\ll
\cW_{w,h}\frac{q^5(1+|t|)^{5/2}}{X^5}
\sum_{\substack{\vv\in\Z^4\\ \vv\ \mathrm{good}}}
\|q\balpha-\vv\|^{-5}\\
&\ll
\cW_{w,h} \frac{q^5(1+|t|)^{5/2}}{X^5}.
\end{aligned}
\end{equation}

By~\eqref{eq: lattice estim}, in the range~\eqref{eq:outer-t-range}, 
the number of bad vectors is
\[
O\(\(1+q|t|/X\)^4\) = O\(\(q|t|/X\)^4\).
\]
Thus~\eqref{eq:I-two-cost-stationary} 
(since we also have $|t| \gg 1$ in the range~\eqref{eq:outer-t-range}) yields
\begin{equation}
\label{eq:outer-bad-vectors-edge}
\sum_{\substack{\vv\in\Z^4\\ \vv\ \mathrm{bad}}}
|I_w(t,\bxi_{\vv})|
\ll
\cW_{w,h} \frac{q^4|t|^2}{X^4}.
\end{equation}
If $q\balpha\in\Z^4$, including the omitted resonant vector among the bad
vectors only enlarges this upper bound.

Returning to the definition of $\mathcal{K}_q^\circ(\boldsymbol{\alpha})$, we apply~\eqref{eq:outer-good-vectors-edge} and~\eqref{eq:outer-bad-vectors-edge} to bound the contribution from the range $|t| \geq c_0 X/q$. Making the substitution $y = q|t|/X$ and using $q \leq X$, we obtain
\begin{align*}
&\int_{|t|\geq c_0X/q}
\(1+\frac{q|t|}{X}\)^{-K}
\sideset{}{^\circ}\sum_{\vv\in\Z^4}
|I_w(t,\bxi_{\vv})|\,dt\\
&\qquad\qquad\ll 
\cW_{w,h} \biggl(
\(\frac qX\)^{3/2}
\int_{c_0}^{\infty}(1+y)^{-K+5/2}\,dy\\
& \qquad\qquad\qquad\qquad\qquad\qquad\qquad+
\frac qX\int_{c_0}^{\infty}y^2(1+y)^{-K}\,dy
\biggr), 
\end{align*}
where the implied constants may depend on $B$ and $K$. 

Since $K \geq 5$, both integrals on the right-hand side converge. 
Consequently,
\begin{equation}
\label{eq:Kq-outer-bound-edge}
\int_{|t|\geq c_0X/q}
\(1+\frac{q|t|}{X}\)^{-K}
\sideset{}{^\circ}\sum_{\vv\in\Z^4}
|I_w(t,\bxi_{\vv})|\,dt
\ll 
\cW_{w,h} \frac qX.
\end{equation}

It now remains to treat the case 
\[
|t|<c_0X/q.
\]

 Choose a nearest lattice point
$\vv_0\in\Z^4$ to $q\balpha$
(choosing arbitrarily  if there are several).  If $\vv\neq\vv_0$, then at least one coordinate of
$q\balpha-\vv$ has absolute value at least $1/2$.  Moreover, for
$\vu\in\supp w$, as before, we have $\| \vu\| \le 2B$ and hence, recalling the definition of 
$\bxi_{\vv}$ in~\eqref{eq:xiv}, we see that
\[
|t|\|\vu M_0\|\leq\frac{X}{4q} \le \frac{X}{2q}\|q\balpha-\vv\| = \frac{1}{2}\|\bxi_{\vv}\| , 
\]
from which we conclude that 
\[
\|t\vu M_0+\bxi_{\vv}\|
\geq \frac{1}{2}\|\bxi_{\vv}\|  = 
\frac{X}{2q}\|q\balpha-\vv\|.
\]
Using~\eqref{eq:I-two-cost-frequency} with  
\[
L=
\frac{X}{2q}
\|q\boldsymbol\alpha-\vv\|,
\]
the lattice-sum estimate in~\eqref{eq: lattice sum} and the bound $(1+q|t|/X)^{-K}\leq1$, we obtain
\begin{equation}
\begin{split}
\label{eq:Kq-inner-nonnearest-edge}
\int_{|t|<c_0X/q}
\sum_{\vv\neq\vv_0}
|I_w(t,\bxi_{\vv})|\,dt & \ll
\cW_{w,h} \frac{q^5}{X^5}
\int_{|t|<c_0X/q}(1+|t|)^{5/2}\,dt \\
&\ll
\cW_{w,h} \(\frac qX\)^{3/2}
\le 
\cW_{w,h} \frac qX.
\end{split}
\end{equation}

If $q\balpha\in\Z^4$, then $\vv_0=q\balpha$ is precisely the
omitted resonant vector, which is excluded from the sum in~\eqref{eq:Kq-inner-nonnearest-edge}.
 The bounds~\eqref{eq:Kq-outer-bound-edge} 
and~\eqref{eq:Kq-inner-nonnearest-edge} then prove the second case of~\eqref{eq:Kq-alpha-bound}.

Suppose now that $q\balpha\notin\Z^4$.
To bound the remaining $\vv_0$-term, we put
\[
\Delta=\distZ{q\balpha}
=\|q\balpha-\vv_0\|
\mand
\Xi=\frac Xq\Delta=\|\bxi_{\vv_0}\|.
\]
We may also extend its $t$-integral to all of $\R$.

If $\Xi \leq1$, then~\eqref{eq:I-two-cost-stationary} gives
\[
\int_{\R}|I_w(t,\bxi_{\vv_0})|\,dt
\ll \cW_{w,h} .
\]
Together with~\eqref{eq:Kq-outer-bound-edge} and~\eqref{eq:Kq-inner-nonnearest-edge}, 
this yields $\cK_q^\circ(\balpha)\ll \cW_{w,h} $, which establishes the first
case of~\eqref{eq:Kq-alpha-bound} for  $\Xi \leq1$.

It remains to consider $\Xi > 1$.    Choose
\[
c_1=(4B)^{-1}.
\]
Then, for $|t|\leq c_1\Xi $ and
$\vu\in\supp w$, we have
\[
\|t\vu M_0+\bxi_{\vv_0}\|
\geq \Xi-2B|t|
\geq\frac {1}{2} \Xi.
\]
It follows from~\eqref{eq:I-two-cost-frequency} upon taking $L=\Xi/2$ that
\begin{equation}
\label{eq:nearest-small-t-edge}
\begin{aligned}
\int_{|t|\leq c_1\Xi}|I_w(t,\bxi_{\vv_0})|\,dt
&\ll 
\cW_{w,h}  \Xi^{-5}\int_0^{c_1\Xi}(1+t)^{5/2}\,dt\\
&\ll 
\cW_{w,h}  \Xi^{-3/2}.
\end{aligned}
\end{equation}
On the complementary range,~\eqref{eq:I-two-cost-stationary} gives
\begin{equation}
\label{eq:nearest-large-t-edge}
\int_{|t|>c_1\Xi}|I_w(t,\bxi_{\vv_0})|\,dt
\ll 
\cW_{w,h}  \Xi^{-1}.
\end{equation}
Trivially, each coordinate of $q\balpha$ lies within $1/2$ of an integer. Therefore 
$\Delta\leq1$, and hence $q/X\leq \Xi^{-1}$.  
Combining~\eqref{eq:Kq-outer-bound-edge}, \eqref{eq:Kq-inner-nonnearest-edge}, \eqref{eq:nearest-small-t-edge} 
and~\eqref{eq:nearest-large-t-edge}, we conclude that
\[
\cK_q^\circ(\balpha)
\ll 
\cW_{w,h} \(\frac qX+\Xi^{-3/2}+\Xi^{-1}\)
\ll 
\cW_{w,h}  \Xi^{-1}.
\]
Since $\Xi>1$, this is equivalent to the first case of~\eqref{eq:Kq-alpha-bound}.

Substituting~\eqref{eq:Kq-alpha-bound} into~\eqref{eq:Tq-to-I-sum}  concludes the proof. 
\end{proof}

We emphasise that the implied constants in Lemma~\ref{lem:nonresonant-dual-frequencies}  are uniform in $X$, $q$, $\balpha$, $w$ and $h$
subject to the conditions above. 

We estimate the second moment of the frequency factor in
Lemma~\ref{lem:nonresonant-dual-frequencies} over vectors $\vb\in(\Z/m\Z)^4$. For an integer $m\geq1$ and $\vb\in(\Z/m\Z)^4$, define
\begin{equation}
\label{eq:K_def}
    K_{q,m}(\vb)=
\begin{cases}
\displaystyle
\left(1+\frac Xq\distZ{q\vb/m}\right)^{-1},
&q\vb\not\equiv\vec0\bmod m,\\[6pt]
q/X,&q\vb\equiv\vec0\bmod m.
\end{cases}
\end{equation}

\begin{lem}  
\label{lem:frequency-grid-second-moment}
Let $X\geq1$ and $q$ be an integer with $1\leq q\leq X$.
Then, uniformly in $X$, $q$ and $m$, we have
\[
\sum_{\vb\in(\Z/m\Z)^4}K_{q,m}(\vb)^2
\ll m^4\left(\frac qX\right)^2.
\]
\end{lem}

\begin{proof} 
Put $g=\gcd(q,m)$, $q_*=q/g$ and $m_*=m/g$.
Since $q_*$ and $m_*$ are coprime, as $\vb$ varies modulo $m$, the point
$q\vb/m\bmod\Z^4$ runs over the grid $m_*^{-1}\Z^4/\Z^4$,
each point occurring $g^4$ times. 

The result is trivial if $m_* =1$, as we  have $q\vb\equiv\vec0\bmod m$ 
for all $\vb\in(\Z/m\Z)^4$. So we now assume that $m_*\ge 2$. 

Let  $\cC_{m_*}=\(\Z\cap[-m_*/2,m_*/2)\)^4$ be the set of
centred representatives for $(\Z/m_*\Z)^4$.
For each $\vb\in(\Z/m\Z)^4$, let $\vc\in \cC_{m_*}$ be the unique
representative of $q_*\vb\bmod m_*$. Since $q/m=q_*/m_*$, we have
\[
\distZ{q\vb/m}=\frac{\|\vc\|}{m_*}.
\]
For $\vc\ne\vec0$, corresponding to the first case of~\eqref{eq:K_def}, we have
\[
K_{q,m}(\vb)
\leq\frac qX\frac{m_*}{\|\vc\|}.
\]

We see from~\eqref{eq:K_def} that  each of the $g^4$ vectors
$\vb\in(\Z/m\Z)^4$ satisfying $q\vb\equiv\vec0\bmod m$ contributes
$(q/X)^2$ to the desired sum.  

Note that
\[
\sum_{\substack{\vc\in \cC_{m_*} \\\vc\ne\vec0}} \frac1{\|\vc\|^2}
=\sum_{1 \le r \le m_* / 2}   \sum_{\substack{\vc\in \cC_{m_*} \\\|\vc\|_{\infty} = r}}
\frac1{\|\vc\|^2}
\ll \sum_{1 \le r \le m_* / 2} r
\ll  m_*^2.
\]
Consequently,
\[
\sum_{\vb\in(\Z/m\Z)^4}K_{q,m}(\vb)^2
\ll g^4\left(\frac{q}{X}\right)^2+ g^4 \sum_{\substack{\vc\in \cC_{m_*}\\\vc\ne\vec0}}\(\frac{q}{X}\frac{m_*}{\|\vc\|}\)^2 \ll m^4\left(\frac qX\right)^2, 
\]
which proves the result.
\end{proof}

\section{Evaluation of the resonant contribution}
\label{sec:resonant}

\subsection{Chinese remaindering}
We now turn to the resonant dual frequency.  For a rational twist
$\balpha=\vb/s$, with $s=\fs(\balpha)$, the resonant vector
$\vv=q\balpha \in\Z^4$ occurs precisely when $s\mid q$ and contributes to
the main term.   We first record the Chinese Remainder Theorem for 
complete exponential sums, which is essentially~\cite[Equation~(12.21)]{IwKow04}.

We use $v_p(t)$ to denote the $p$-adic valuation of $t \in \Z$.
 
\begin{lemma}
\label{lem:resonant-crt-factorization}
Let $\balpha=\vb/s\in\Q^4$, where $s=\fs(\balpha)$. For each prime
$p$, put
\[
a_p=v_p(s).
\]
\begin{itemize}
\item If $a_p>0$, put
\[
s_p=\frac{s}{p^{a_p}},
\]
choose $\overline{s}_p\in\Z$ such that
$
s_p\overline{s}_p\equiv1\bmod{p^{a_p}}
$, 
and define
\[
\vb_p\equiv\overline{s}_p\vb\bmod{p^{a_p}},
\qquad
\vb_p\in(\Z/p^{a_p}\Z)^4.
\]

\item 
If $a_p=0$, put $\vb_p=\vec{0}$ and interpret the corresponding
additive character as trivial.
\end{itemize}
Let $q\geq1$. Then
\[
q\balpha\in\Z^4
\quad\Longleftrightarrow\quad
s\mid q, 
\]
and if $s\mid q$,  then, with $\nu_p=v_p(q)$, we have
\[
S(q,q\balpha)
=
\prod_{p\mid q}
S\left(p^{\nu_p},p^{\nu_p-a_p}\vb_p\right),
\]
where $p^{\nu_p-a_p}\vb_p$ is naturally regarded as a vector modulo
$p^{\nu_p}$.

\end{lemma}

In the notation of
Lemma~\ref{lem:resonant-crt-factorization}, the additive character
decomposes as
\begin{equation}
\label{eq:character-crt-decomposition}
\e_s(\vb\cdot\vx)
=
\prod_{p\mid s}
\e_{p^{a_p}}(\vb_p\cdot\vx_p),
\end{equation}
where $\vx_p$ denotes the reduction of
$\vx\in(\Z/s\Z)^4$ modulo $p^{a_p}$. When $s=1$, the empty product
is interpreted as $1$.

We also observe the following identity.

\begin{lemma} 
\label{lem: Spab}
For  any prime $p$,  a vector   $\vb = (b_1, \ldots, b_4) \in\Z^4$ with $\gcd (b_1, \ldots, b_4, p)  =1$ and an integer $a \ge 1$, we have 
\[
S\(p^{a},\vb\) = p^a \sum_{g\in \cG_{p^a}}
\e_{p^{a}}(\vb\cdot g).
\]
\end{lemma}

\begin{proof} We have 
\begin{equation}
\label{eq:S S1S2}
S\(p^{a},\vb\)  = \Sigma_1 - \Sigma_2,
\end{equation}
where 
\begin{equation}
\label{eq:S1}
 \Sigma_1
=\sum_{c \bmod p^a }  \sum_{\vx\bmod p^a} 
\e_{p^a}\(cQ(\vx)+\vb\cdot\vx\) = p^a \sum_{g\in \cG_{p^a}}
\e_{p^{a}}(\vb\cdot g),
\end{equation}
and 
\begin{align*}
\Sigma_2 & =  \sum_{c=0}^{p^{a-1} -1}  \sum_{\vx\bmod p^a}
\e_{p^a}\(cpQ(\vx)+\vb\cdot\vx\)\\ 
& = \sum_{c \bmod p^{a-1} }  \sum_{\vx\bmod p^a}
\e_{p^{a-1}}\(cQ(\vx)\) \e_{p^a}\(\vb\cdot\vx\).
\end{align*}
To conclude the proof, it is enough to show that $\Sigma_2 = 0$.
For $a=1$, this is obvious as we have a complete sum over $\vx \in \(\Z/p\Z\)^4$. 
For $a \ge 2$, we write $\vx = \vy + p^{a-1} \vz$ with   $\vy \in \(\Z/p^{a-1}\Z\)^4$
and $\vz \in \(\Z/p\Z\)^4$. Then
\begin{align*}
\Sigma_2 
 & = \sum_{c \bmod p^{a-1} }   \sum_{\vy \bmod p^{a-1}} \\
& \qquad \qquad \qquad \quad \sum_{\vz \bmod p}
\e_{p^{a-1}}\(cQ( \vy + p^{a-1} \vz)\) \e_{p^a}\(\vb\cdot\( \vy + p^{a-1} \vz\)\)\\
 & = \sum_{c \bmod p^{a-1} }   \sum_{\vy \bmod p^{a-1}} \e_{p^{a-1}}\(cQ( \vy)\) \e_{p^a}\(\vb\cdot \vy \) \sum_{\vz \bmod p}
\e_{p}\(\vb\cdot\ \vz\) = 0. 
\end{align*} 
Recalling~\eqref{eq:S S1S2} and~\eqref{eq:S1}, we conclude the proof. 
\end{proof}

\subsection{Singular series}
We now evaluate a  series involving the sums $S(q,q\balpha)$ 
which contributes to the singular product~\eqref{eq:sq-singular-product-intro}.

\begin{lemma} 
\label{lem:resonant-singular-series}
Let $\balpha=\vb/s\in\Q^4$, where $s=\fs(\balpha)$.  Then the series
\[
\mathfrak S(\balpha)
=
\sum_{\substack{q\geq1\\q\balpha\in\Z^4}}
\frac{S(q,q\balpha)}{q^4}
\]
converges absolutely and satisfies
\[
\mathfrak S(\balpha)
=
\frac{\lambda(\balpha)}{\zeta(2)}, 
\]
where $\lambda(\balpha)$ is  given by~\eqref{eq:lambda-s-b}.
\end{lemma}

\begin{proof}
By Lemma~\ref{lem:complete-exponential-sum} and the  bound~\eqref{eq: tau} on $\tau(q)$, we have 
\[
\sum_{\substack{q\geq1\\q\balpha\in\Z^4}}
\frac{|S(q,q\balpha)|}{q^4}
\ll
\sum_{q\geq1}\frac{\tau(q)}{q^{3/2}}
<\infty,
\]
so the series converges absolutely.

By Lemma~\ref{lem:resonant-crt-factorization}, the condition $q\boldsymbol{\alpha} \in \mathbb{Z}^4$ is equivalent to $s \mid q$. Consequently, if $a_p = v_p(s)$, then the exponent $\nu_p = v_p(q)$ satisfies $\nu_p \geq a_p$. Lemma~\ref{lem:resonant-crt-factorization} also provides the corresponding factorization of $S(q, q\boldsymbol{\alpha})$ over prime powers.
Absolute convergence therefore gives
\begin{equation}
\label{eq:S and Sp}
\mathfrak S(\balpha)
=
\prod_p \mathfrak S_p(\balpha)
\end{equation}
with 
\begin{equation}
\label{eq:Sp}
    \mathfrak S_p(\balpha) =
\sum_{\nu\geq a_p} p^{-4\nu} S\(p^\nu,p^{\nu-a_p}\vb_p\),
\end{equation}
where, when $a_p=0$, the term with $\nu=0$ is interpreted as $1$.

 If $a_p>0$, then the minimality of $s=\fs(\balpha)$ implies that the
components of $\vb_p$ are not all divisible by $p$.

For $\nu\geq a_p$, define
\begin{equation}
\label{eq:B-p-nu-def}
B_{p,\nu}
=
p^{-3\nu}
\sum_{g\in \cG_{p^\nu}}
\e_{p^{a_p}}(\vb_p\cdot g).
\end{equation}
Here, by~\eqref{eq:def_G_m},
\[
\cG_{p^\nu}=\SL_2(\Z/p^\nu\Z),
\]
and in the dot product $\vb_p\cdot g$, the matrix $g$ is identified
with the vector of its four entries. When $a_p=0$, the additive
character in~\eqref{eq:B-p-nu-def} is interpreted as trivial. 

For $m\geq1$ and $n\in\Z$, let
\[
c_m(n)=\sideset{}{^*} \sum_{r\bmod m}\e_m(rn) 
=\mu\(m/\gcd(m,n)\)\frac{\varphi(m)}{\varphi\(m/\gcd(m,n)\)}
\]
denote the Ramanujan sum; see~\cite[Equation~(3.3)] {IwKow04}.  For every $\nu\geq1$, a direct computation gives
\[
c_{p^\nu}(n)
=
p^\nu\one_{p^\nu\Z}(n)
-
p^{\nu-1}\one_{p^{\nu-1}\Z}(n).
\]
Note that for $\nu\geq\max\{1,a_p\}$, the definition~\eqref{eq:S-q-v-def}  can be rewritten as 
\[
S(p^\nu,p^{\nu-a_p}\vb_p)
=
\sum_{\vx\bmod p^\nu}
\e_{p^{a_p}}(\vb_p\cdot\vx)c_{p^\nu}(Q(\vx)).
\]
Furthermore, if $\nu > a_p$, then 
\begin{align*}
S(p^\nu,p^{\nu-a_p}\vb_p)
& =
p^\nu \sum_{\vx\bmod p^\nu}
\e_{p^{a_p}}(\vb_p\cdot\vx)\one_{p^\nu\Z}  (Q(\vx))\\
& \qquad \qquad - p^{\nu-1}   \sum_{\vx\bmod p^\nu}
\e_{p^{a_p}}(\vb_p\cdot\vx)\one_{p^{\nu-1}\Z}  (Q(\vx))\\
& =
p^{4\nu} B_{p,\nu} - p^{\nu-1}   \sum_{\vx\bmod p^\nu}
\e_{p^{a_p}}(\vb_p\cdot\vx)\one_{p^{\nu-1}\Z}  (Q(\vx)).
\end{align*}
Since $\nu-1 \ge a_p$,  we also have 
\begin{align*}
\sum_{\vx\bmod p^\nu}
\e_{p^{a_p}}(\vb_p\cdot\vx)&\one_{p^{\nu-1}\Z}  (Q(\vx)) \\
&
= p^4 \sum_{\vx\bmod p^{\nu-1}}
\e_{p^{a_p}}(\vb_p\cdot\vx)\one_{p^{\nu-1}\Z}  (Q(\vx))\\
& = p^4 p^{3(\nu-1)}  B_{p,\nu-1} = p^{3\nu+1}  B_{p,\nu-1}.
\end{align*}
Hence, if $\nu > a_p$, we have 
\begin{equation}
\label{eq: S B-B}
p^{-4\nu}S(p^\nu,p^{\nu-a_p}\vb_p)
=
B_{p,\nu}-B_{p,\nu-1}.
\end{equation}

It is now convenient to extend the definition of  $\lambda(\balpha)$ in~\eqref{eq:lambda-s-b}.
For any integer $t\geq 1$
and any vector $\vb\in (\Z/t\Z)^4$ (not necessarily arising from $\balpha$),
define
\begin{equation}
\label{eq: kappa tb}
\kappa_t(\vb)
=
\frac{1}{\#\cG_t}\sum_{g\in\cG_t}\e_t(\vb\cdot g).
\end{equation}

Next, we consider separately the cases $a_p>0$ ($p\mid s$) and
$a_p=0$ ($p\nmid s$).

For $\nu\geq a_p>0$, the reduction map
\[
\cG_{p^\nu}\longrightarrow \cG_{p^{a_p}}
\]
is surjective and, by~\eqref{eq:order_Gm}, every fibre has cardinality
$p^{3(\nu-a_p)}$. Hence
\begin{equation}
\label{eq: B const}
B_{p,\nu}
=
p^{-3a_p}\sum_{g\in\cG_{p^{a_p}}}
\e_{p^{a_p}}(\vb_p\cdot g)
=
(1-p^{-2})\kappa_{p^{a_p}}(\vb_p).
\end{equation}
In particular, $B_{p,\nu}$ is independent of $\nu$. Thus,
by~\eqref{eq: S B-B},
\[
p^{-4\nu}S\(p^\nu,p^{\nu-a_p}\vb_p\)=0
\qquad (\nu>a_p>0).
\]

Since all terms with $\nu>a_p$ vanish, Lemma~\ref{lem: Spab},
applied with $a=a_p$, gives
\[
\mathfrak S_p(\balpha)
=
p^{-4a_p}S\(p^{a_p},\vb_p\)
=
B_{p,a_p}.
\]
Consequently,~\eqref{eq: B const} yields
\begin{equation}
\label{eq:SingSer ap>0}
\mathfrak S_p(\balpha)
=
(1-p^{-2})\kappa_{p^{a_p}}(\vb_p).
\end{equation}

We now consider the case $a_p=0$, that is, $p\nmid s$. In this case,
\[
\mathfrak S_p(\balpha)
=
\sum_{\nu\geq0}
p^{-4\nu}S \(p^\nu,p^\nu\vb_p\).
\]
The term corresponding to $\nu=0$ is $1$. Likewise, since
$\cG_1$ consists of a single element, we have
\[
B_{p,0}=1.
\]

For every $\nu\geq1$, the character in~\eqref{eq:B-p-nu-def} is trivial. Hence,
by~\eqref{eq:order_Gm},
\[
B_{p,\nu}
=
p^{-3\nu}\#\cG_{p^\nu}
=
1-p^{-2}.
\]
In particular,
\[
B_{p,\nu}-B_{p,\nu-1}=0
\qquad (\nu\geq2).
\]
Substituting~\eqref{eq: S B-B} into~\eqref{eq:Sp}, we therefore obtain 
\begin{equation}
\label{eq:SingSer ap=0}
\mathfrak S_p(\balpha)
=
1+B_{p,1}-B_{p,0}
=
1-p^{-2}.
\end{equation}

Combining~\eqref{eq:S and Sp} with~\eqref{eq:SingSer ap>0} 
and~\eqref{eq:SingSer ap=0},  we obtain
\begin{align*}
\mathfrak S(\balpha)
&=
\prod_{p\nmid s}(1-p^{-2})
\prod_{p^{a_p}\| s}\((1-p^{-2})\kappa_{p^{a_p}}(\vb_p)\)\\
& =
\prod_p(1-p^{-2}) 
\prod_{p^{a_p}\| s}\kappa_{p^{a_p}}(\vb_p) = \zeta(2)^{-1} \prod_{p^{a_p}\| s}\kappa_{p^{a_p}}(\vb_p) ,
\end{align*}
where $a_p=v_p(s)$ and $\zeta$ denotes the Riemann zeta function.

By the Chinese Remainder Theorem (more precisely, using~\eqref{eq:character-crt-decomposition}),
we obtain  
\[
\kappa_s(\vb) =
\prod_{p^{a_p}\| s}\kappa_{p^{a_p}}(\vb_p)
\]
 (which is  a full analogue of Lemma~\ref{lem:resonant-crt-factorization}). 
Therefore
\[
\mathfrak S(\balpha)
=
\zeta(2)^{-1}\kappa_s(\vb) = \zeta(2)^{-1} \lambda(\balpha),
\]
as required.
\end{proof}

In particular,  we see from~\eqref{eq:lambda-s-b} that $\lambda(\mathbf{0})  =1$ and thus Lemma~\ref{lem:resonant-singular-series} yields
\begin{equation}
\label{eq:Sa S0}
\fS(\balpha)=\lambda(\balpha)\mathfrak S(\vec{0}).
\end{equation}

\subsection{Asymptotic evaluation of  the resonant contribution}

We now evaluate the resonant contribution in the
Poisson expansion~\eqref{eq:delta-poisson-formula}, namely
\[
\cM_w(X;\balpha)
=
X^2\sum_{\substack{q\leq X\\q\balpha\in\Z^4}}
\frac{S(q,q\balpha)}{q^4}J_{w,q}(X),
\]
where
\[
J_{w,q}(X)
=
X^2\int_{\R}
p_{q,X}(z)I_w(zX^2,\vec{0})\,dz.
\]
Here $\vec{0}\in\R^4$ is the zero vector. 

 We recall our definitions of $N_w(X)$ in~\eqref{eq:Nw} and of $ \cW_{w,h}$ in~\eqref{eq:Whw}.

\begin{lem} 
\label{lem:rational-twisted-smooth-count} Assume that $B \ge 1$ is fixed.
Let $X \ge 1$ and  let $w\in \cD^\infty(\R^4)$ be supported in
$[-B,B]^4$. Then for $\balpha\in\Q^4$, with
$s=\fs(\balpha)$, for every $0<h\leq1$, we have 
 \[
\left| \cM_w(X;\balpha)- \lambda(\balpha)N_w(X) \right| \le \cW_{w,h}   X^{3/2+o(1)}.
\] 
\end{lem}

\begin{proof}
Put
\[
I_{w,0}(t)
=
\int_{\R^4}w(\vu)\e(tQ_0(\vu))\,d\vu.
\]
By~\eqref{eq:Q1-def} and~\eqref{eq:I-z-v-def}, 
\[
I_w(t,\vec{0})=\e(-t/X^2)I_{w,0}(t). 
\]

By Corollary~\ref{cor:oscillatory-integral}~(ii), we have
\begin{equation}
\label{eq:I0-stationary-phase}
I_{w,0}(t) 
\ll 
\cW_{w,h}  (1+|t|)^{-2}.
\end{equation}
Consequently, the singular integral
\[
\sigma_\infty(w)
=
\int_{\R}I_{w,0}(t)\,dt
\] 
converges absolutely and satisfies
\begin{equation}
\label{eq:sigma-infinity-weight-bound}
\sigma_\infty(w)\ll \cW_{w,h}  .
\end{equation}

The change of variables $t=zX^2$ gives
\[
J_{w,q}(X)
=
\int_{\R}p_{q, X}(t/X^2)\e(-t/X^2)I_{w,0}(t)\,dt.
\]

We now claim that for every fixed $0<\eta<1/2$,
\begin{equation}
\label{eq:Jq-resonant-approx}
J_{w,q}(X)
=
\sigma_\infty(w)
+O\(
\cW_{w,h}  \(\frac qX\)^{1/2-\eta}
\),
\qquad 1\leq q\leq X, 
\end{equation}
where the implied constant may also depend on $\eta$.
Indeed, write
\[
J_{w,q}(X)-\sigma_\infty(w)
=
\int_{\R}
\(p_{q, X}(t/X^2)\e(-t/X^2)-1\)I_{w,0}(t)\,dt,
\]
and split the integral at
\[
T=(X/q)^{1/2-\eta} \ge 1.
\]
For $|t|\leq T$, we use
\begin{align*}
p_{q, X}(t/X^2)&\e(-t/X^2)-1\\
&=
\(p_{q, X}(t/X^2)-1\)\e(-t/X^2)
+\(\e(-t/X^2)-1\).
\end{align*}
Since $T\geq1$, taking $K= 3\eta^{-1}/4 \ge 1$ in the asymptotic formula  for $p_{q, X}(z)$ 
of Lemma~\ref{lem:delta-exp} gives
\[
p_{q, X}(t/X^2)-1
\ll 
\(\frac qX\)^K T^{2K+2}
=
\(\frac qX\)^{1/2+2\eta} \le  \frac{1}{T}.
\]
Moreover,
\[
|\e(-t/X^2)-1|\ll |t|X^{-2}.
\]
Hence, 
\[
p_{q, X}(t/X^2)\e(-t/X^2)-1 \ll T^{-1} + T X^{-2} \ll T^{-1}.
\]

By~\eqref{eq:I0-stationary-phase}, 
we estimate the contribution from $|t|\leq T$ as
\begin{align*}
\int_{|t|\leq T} \(p_{q, X}(t/X^2)\e(-t/X^2)-1\)I_{w,0}(t)\,dt 
 \ll \cW_{w,h}   T^{-1}. 
\end{align*}

Finally, for $|t|>T$, 
 the crude bound $p_{q, X}(z) \ll 1$ from Lemma~\ref{lem:delta-exp} and~\eqref{eq:I0-stationary-phase} give
\begin{align*}
\int_{|t|>T} 
\(|p_{q, X}(t/X^2)|+1\) |I_{w,0}(t)|\,dt &
 \ll \cW_{w,h}   \int_{|t|>T}  (1+|t|)^{-2} \,dt\\
& \ll \cW_{w,h}   T^{-1}, 
\end{align*}
which proves~\eqref{eq:Jq-resonant-approx}.

Substituting~\eqref{eq:Jq-resonant-approx} into the definition of
$\cM_w(X;\balpha)$ and applying
Lemma~\ref{lem:complete-exponential-sum} with $\vv=q\balpha$
(which is valid because $q\balpha\in\Z^4$), together with the divisor
bound~\eqref{eq: tau}, we obtain
\[
\cM_w(X;\balpha)
=
\sigma_\infty(w)X^2
\sum_{\substack{q\leq X\\ q\balpha\in\Z^4}}
\frac{S(q,q\balpha)}{q^4}
+
O\!\left(
\cW_{w,h}X^{3/2+\eta}
\right).
\]
 
Applying Lemma~\ref{lem:complete-exponential-sum} and the divisor
bound~\eqref{eq: tau} once more gives
\[
\sum_{\substack{q\leq X\\ q\balpha\in\Z^4}}
\frac{S(q,q\balpha)}{q^4}
=
\mathfrak S(\balpha)
+
O\!\left(X^{-1/2+o(1)}\right).
\]

Combining this estimate with the preceding formula for
$\cM_w(X;\balpha)$ and using~\eqref{eq:sigma-infinity-weight-bound}, we obtain
\begin{equation}
\label{eq:resonant-main-extraction}
\cM_w(X;\balpha)
=
\sigma_\infty(w)X^2\mathfrak S(\balpha)
+
O\!\left(
\cW_{w,h}X^{3/2+\eta}
\right).
\end{equation}

Applying~\eqref{eq:resonant-main-extraction} at $\balpha$ and $\vec0$,
and using~\eqref{eq:Sa S0} together with $|\lambda(\balpha)|\leq1$,
we obtain, for every fixed $0<\eta<1/2$,
\[
\cM_w(X;\balpha)
=\lambda(\balpha)\cM_w(X;\vec0)
+O_\eta\!\left(\cW_{w,h}X^{3/2+\eta}\right).
\]
It therefore remains to compare $\cM_w(X;\vec0)$ with $N_w(X)$.

For $\balpha=\vec0$, the unique resonant dual frequency
in~\eqref{eq:delta-poisson-formula} is $\vv=\vec0$
for every $q\leq X$.
The second case of Lemma~\ref{lem:nonresonant-dual-frequencies}
therefore gives
\[
\begin{aligned}
\sum_{q\leq X}|\cT_{w,q}^\circ(\vec0)|
&\ll \cW_{w,h}X\sum_{q\leq X}\frac{\tau(q)}{q^{1/2}}\\
&\ll \cW_{w,h}X^{3/2+o(1)}.
\end{aligned}
\]
Since $T_w(X;\vec0)=N_w(X)$, separating the resonant contribution
in~\eqref{eq:delta-poisson-formula} yields
\[
\begin{aligned}
\cM_w(X;\vec0)
&=T_w(X;\vec0)-\sum_{q\leq X}\cT_{w,q}^\circ(\vec0)
  +O(\|w\|_\infty)\\
&=N_w(X)+O\!\left(\cW_{w,h}X^{3/2+o(1)}\right),
\end{aligned}
\]
where we used $\|w\|_\infty\leq\cW_{w,h}$.

Combining these two comparisons, using $|\lambda(\balpha)|\leq1$,
and recalling that $\eta>0$ is arbitrary, we conclude the proof.
\end{proof}

\section{Proof of Theorem~\ref{thm:main-squarefree}}

\subsection{Counting lifts of $\SL_2(\Z/m\Z)$ matrices with smooth cut-offs}

Recall the definitions of $\cW_{w,h}$ in~\eqref{eq:Whw} and $A_{w,m}(X)$ in~\eqref{eq:A_dw}.

\begin{lem}
\label{lem:SmoothApprox} 
Let $B\geq1$ be fixed. Uniformly for real $X\geq1$, integers $m\ge1$, real $h\in(0,1]$ and the test functions $w\in\cD^\infty(\R^4)$
supported in $[-B,B]^4$, we have
\[
A_{w,m}(X)=\beta_mN_w(X)
+O\!\left(\cW_{w,h} mX^{3/2+o(1)}\right).
\]
\end{lem}

\begin{proof}
For $m>X$, the bounds $0\leq\beta_m\leq1$ and
$|A_{w,m}(X)|+|N_w(X)|\ll\|w\|_\infty X^2$ prove the assertion.
We may therefore assume that $m\leq X$, so that factors $m^{o(1)}$
can be absorbed into $X^{o(1)}$.  Define
\[
\cS_m=\{\vx\bmod m:~Q(\vx)\equiv0\bmod m,
                       ~\Rform(\vx)\equiv0\bmod m\},
\]
and let $\Psi_m=\one_{\cS_m}$ on $(\Z/m\Z)^4$, with Fourier transform
\[
\widehat\Psi_m(\vb)
=\sum_{\vx\bmod m}\Psi_m(\vx)\e_m(-\vb\cdot\vx).
\]
For every integral solution of $Q(\va)=0$, Fourier inversion gives
\[
\one_{m\Z}\(\Rform(\va)\)
=
m^{-4}\sum_{\vb\in(\Z/m\Z)^4}
\widehat{\Psi}_m(\vb)\e_m(\vb\cdot\va).
\]
Consequently, recalling the definition~\eqref{eq:TwX}, we obtain
\begin{equation}
\begin{split}
\label{eq:Adw-fourier}
A_{w,m}(X)& =
\sum_{\substack{\va\in\Z^4\\Q(\va)=0}}
 w(\va/X)\one_{m\Z}(\Rform(\va))\\
& =m^{-4}\sum_{\vb\in(\Z/m\Z)^4}
\widehat\Psi_m(\vb)T_w(X;\vb/m).
\end{split} 
\end{equation}

For each $\vb\in(\Z/m\Z)^4$, put $s=\fs(\vb/m)$ and
$\vb^*=s\vb/m\in\Z^4$.  Since $s\mid m$ and reduction
$\cG_m\to\cG_s$ is surjective with equal-sized fibres, we have
\begin{equation}
\label{eq:kappa-lambda-fourier-mode}
\kappa_m(\vb)=\kappa_s(\vb^*)=\lambda(\vb/m).
\end{equation}

Using Fourier inversion, together with~\eqref{eq: kappa tb}, we obtain
\begin{align*}
m^{-4}\sum_{\vb\in(\Z/m\Z)^4}
\widehat{\Psi}_m(\vb)\kappa_m(\vb)
=
\frac{1}{\#\cG_m}
\sum_{g\in\cG_m}\Psi_m(g)
=
\frac{\#\cS_m}{\#\cG_m},
\end{align*}
where, as before, we identify $g\in\cG_m$ with a vector $\vx \in \(\Z/m\Z\)^4$ of 
its coefficients. 
Hence, using  Corollary~\ref{cor:beta_d2}, we obtain 
\begin{equation}
    \label{eq:sim_coe_mian term}
    m^{-4}\sum_{\vb\in(\Z/m\Z)^4}
\widehat\Psi_m(\vb) \kappa_m(\vb)
=
\beta_{m}.
\end{equation}
Since by  the Parseval identity
\[
\sum_{\vb\in(\Z/m\Z)^4}|\widehat\Psi_m(\vb)|^2=m^4\#\cS_m,
\]
the  Cauchy--Schwarz inequality yields
\begin{equation}
\label{eq:fourier-parseval-average}
m^{-4}\sum_{\vb\in(\Z/m\Z)^4}|\widehat\Psi_m(\vb)|
\leq(\#\cS_m)^{1/2}\leq m^{1+o(1)},
\end{equation}
where the last inequality follows from
$\#\cS_m=\beta_m\#\cG_m\leq m^{2+o(1)}$, by
Lemma~\ref{lem:beta-d-estimates} and~\eqref{eq:order_Gm}.

We average over $\vb$ before summing over $q$, retaining
the frequency dependence in Lemma~\ref{lem:nonresonant-dual-frequencies}.
Recall the definition of $K_{q,m}(\vb)$ in~\eqref{eq:K_def}.
Combining Lemma~\ref{lem:frequency-grid-second-moment} with~\eqref{eq:fourier-parseval-average}, we obtain
\begin{equation}
\label{eq:weighted-frequency-average}
\begin{split}
m^{-4}\sum_{\vb\in(\Z/m\Z)^4}
|\widehat\Psi_m(\vb)|K_{q,m}(\vb)
&\ll(\#\cS_m)^{1/2}\frac qX\\
&\leq m^{1+o(1)}\frac qX.
\end{split}
\end{equation}
Lemma~\ref{lem:nonresonant-dual-frequencies}, with 
$
\boldsymbol\alpha=\mathbf b/m,
$ 
gives
\[
\left| \mathcal{T}^\circ_{w,q}(\mathbf{b}/m) \right|
\ll
\mathcal{W}_{w,h} X^2 \frac{\tau(q)}{q^{3/2}} K_{q,m}(\mathbf{b}).
\]
Hence by~\eqref{eq:weighted-frequency-average},
 we obtain
\begin{equation}
\label{eq:averaged-nonresonant-error}
\begin{split}
m^{-4}\sum_{\vb\in(\Z/m\Z)^4}|& \widehat\Psi_m(\vb)|
\sum_{q\leq X}|\cT_{w,q}^\circ(\vb/m)|\\
& \ll\cW_{w,h}m^{1+o(1)}X
\sum_{q\leq X}\frac{\tau(q)}{q^{1/2}}\\
& \leq\cW_{w,h}mX^{3/2+o(1)}.
\end{split}
\end{equation}

Finally, Lemmas~\ref{lem:T approx} and
\ref{lem:rational-twisted-smooth-count}, together
with~\eqref{eq:kappa-lambda-fourier-mode}, give
\[
\begin{split}
T_w(X;\vb/m)-\kappa_m(\vb)N_w(X)
&=\sum_{q\leq X}\cT_{w,q}^\circ(\vb/m) +O\!\left(\cW_{w,h}X^{3/2+o(1)}\right).
\end{split}
\]

Since $\|w\|_\infty\leq\cW_{w,h}$, the error in
Lemma~\ref{lem:T approx} is absorbed into the desired error term for $A_{w,m}(X)$. The resonant error is uniform in $\vb$, and its Fourier average is
$O(\cW_{w,h}mX^{3/2+o(1)})$ by~\eqref{eq:fourier-parseval-average}.
Substituting into~\eqref{eq:Adw-fourier} and
using~\eqref{eq:sim_coe_mian term} and~\eqref{eq:averaged-nonresonant-error}
complete the proof.
\end{proof}

\subsection{Counting lifts of $\SL_2(\Z/m\Z)$ matrices with sharp cut-offs}

We require a simple estimate for the number of matrices $\gamma \in \Gamma$ whose norm $\|\gamma\|_\infty$ lies within a short interval. Specifically, let
\[
  N(X,H) = N(X+H) - N(X-H).
\]

\begin{lemma} 
\label{lem:sharp-box-boundary-shell}
Uniformly for $X\geq2$ and $1\leq H\leq X/2$,
\[
N(X,H)  \ll XH.
\]
\end{lemma}

\begin{proof}
It is enough to count matrices for which the absolute value of one prescribed
entry lies in $[X-H,X+H]$. Without loss of generality,  we may therefore suppose that
$X-H\leq|a_1|\leq X+H$ (since the other cases are fully analogous).

Fix $a_1$ and $a_2$.  The congruence
\[
a_2a_3\equiv-1\bmod{|a_1|}
\]
has no solution unless $\gcd(a_1,a_2)=1$, and otherwise it determines one residue
class for $a_3$ modulo $|a_1|$.  Since $|a_3|\leq X+H$ and
$|a_1|\asymp X$, there are $O(1)$ possible values of $a_3$; the entry
$a_4=(1+a_2a_3)/a_1$ is then determined.  There are $O(X)$ choices for
$a_2$ and $O(H)$ choices for $a_1$, proving the desired result. 
\end{proof}

We note  that the asymptotic formula~\eqref{eq:Krieg} instantly gives a slightly 
weaker bound \[
N(X,H)\ll XH+X\log X,
\] instead of that of Lemma~\ref{lem:sharp-box-boundary-shell}, 
which is also sufficient for our purpose.

\begin{lem}
\label{lem:circle-square-divisor-count}
Uniformly for integers $m$ satisfying $1\leq m\leq X^{1/4}$, we have
\[
    A_m(X) = \beta_mN(X) + O\!\left(m^{-3/5}X^{19/10+o(1)}\right)
\] 
as $X\to\infty$.
\end{lem}

\begin{proof}
Let $1\leq H\leq X/2$ and put
\[
    h=\frac{H}{X}.
\]
Let $w_{h,-}$ and $w_{h,+}$ be the product weights provided by Lemma~\ref{lem:adjustable-edge-cutoffs}~(ii). Their defining properties give
\[
    A_{w_{h,-},m}(X) \leq A_m(X) \leq A_{w_{h,+},m}(X)
\]
and
\[
    N_{w_{h,-}}(X) \leq N(X) \leq N_{w_{h,+}}(X).
\]
Moreover,
\[
    \supp w_{h,\pm}\subseteq[-3/2,3/2]^4, \qquad \|w_{h,\pm}\|_{\cE_h^5}\ll1.
\]

By Lemma~\ref{lem:sharp-box-boundary-shell},
\begin{equation}
\label{eq:smoothing-boundary-count}
\begin{split}
N_{w_{h,+}}(X)&-N_{w_{h,-}}(X)\\
& \ll   \#\{\gamma\in\Gamma:~X-H<\|\gamma\|_\infty\leq X+H\} \ll XH.
\end{split}
\end{equation}
Since
\[
    \cW_{w_{h,\pm},h} = h^{-4}\|w_{h,\pm}\|_{\cE_h^5} \ll \left(\frac{X}{H}\right)^4,
\]

Lemma~\ref{lem:SmoothApprox}, applied to the two weights $w_{h,-}$ and $w_{h,+}$, 
and~\eqref{eq:smoothing-boundary-count} yield
\[
 A_m(X) = \beta_mN(X) + O\!\left( mX^{3/2+o(1)} \left(\frac{X}{H}\right)^4 + \beta_mXH \right).
\]

By Lemma~\ref{lem:beta-d-estimates},
\[
    \beta_m\leq m^{-1+o(1)}.
\]
We now choose
\[
H=\frac12m^{2/5}X^{9/10}.
\]
The assumption $m\leq X^{1/4}$ ensures that $H\leq X/2$, while $H\geq1$ for all sufficiently large $X$.
With this choice, the error terms are bounded by $m^{-3/5}X^{19/10+o(1)}$.
This completes the proof.
\end{proof}

In particular, we see that Lemma~\ref{lem:circle-square-divisor-count} implies:

\begin{cor}
 \label{cor:small-divisor-total-error}
Let $1\leq D\leq X^{1/8}$.   As $X\to \infty$, 
\[
\sum_{d\leq D}
|\mu(d)|\,|A_{d^2}(X)-\beta_{d^2}N(X)|
\le  X^{19/10+o(1)}.
\]
\end{cor}

\subsection{Large square divisors}

We now estimate the contribution to $S_{\mathrm{sq}}(X)$ from $d>D$. 
For this, we use the following simple bound.

\begin{lemma} 
\label{lem:fibre-bound}
We have 
\[
\#\{\gamma\in\Gamma:~\Rform(\gamma)=n\} \le  n^{o(1)}, \qquad \text{as}\ n \to \infty.
\]
\end{lemma}

\begin{proof}
For $n \in \Z$, let 
\[
r_2(n)=\#\{(x,y)\in\Z^2:~x^2+y^2=n\}.
\]
Then,  arguing as in~\cite[Equations~(1.12)--(1.14)]{FrIw09}, we see that 
\[
\#\{\gamma\in\Gamma:~\Rform(\gamma)=n\} \ll r_2(n-2)r_2(n+2).
\]

Since $r_2(m) \ll \tau(m)$ (see, for example,~\cite[Equation~(3.1)]{FrIw09}), 
the result now follows from~\eqref{eq: tau}. 
\end{proof}

Recall the definition of $\Gamma_X$ in~\eqref{eq:Gamma X}.

\begin{lemma} 
\label{lem:large-square-divisor-tail}
Let $D\ge1$. We have
\[
\sum_{d>D} 
\#\{\gamma\in\Gamma_X:~d^2\mid\Rform(\gamma)\}
\le  D^{-1} X^{2+o(1)}.
\]
\end{lemma}

\begin{proof}
If $\gamma\in\Gamma_X$, then $\Rform(\gamma)\le4X^2$.  Therefore
\begin{align*}
\sum_{d>D} 
\#\{\gamma&\in\Gamma_X:~d^2\mid\Rform(\gamma)\}\\
&  \le 
\sum_{d>D}\ \sum_{1\le k\le 4X^2/d^2}
\#\{\gamma\in\Gamma:~\Rform(\gamma)=kd^2\}.
\end{align*}
By Lemma~\ref{lem:fibre-bound}, we now have 
\begin{align*}
\sum_{d>D}\ \sum_{1\le k\le 4X^2/d^2}&
\#\{\gamma\in\Gamma:~\Rform(\gamma)=kd^2\}\\
&\le X^{o(1)} \sum_{d>D}\frac{X^2}{d^2} 
\le  
 D^{-1} X^{2+o(1)}, 
\end{align*}
and the result follows. 
\end{proof}

\subsection{Completion of the proof}
Let $D=\lfloor X^{1/10}\rfloor$.
Splitting the sum over $d$ in~\eqref{eq:incl/excl} 
at $D$, using 
Corollary~\ref{cor:small-divisor-total-error} for
$d\le D$  and  Lemma~\ref{lem:large-square-divisor-tail} to estimate the contribution to 
$S_{\mathrm{sq}}(X)$ from $d > D$, we derive 
\begin{equation}
\label{eq:mobius-truncation}
\begin{split}
S_{\mathrm{sq}}(X)
&=
\sum_{\substack{d\leq D\\d\ \mathrm{squarefree}}}
\mu(d)A_{d^2}(X)\\
&\qquad\qquad +
O\(
\sum_{d>D}|\mu(d)|
\#\{\gamma\in\Gamma_X:~d^2\mid\Rform(\gamma)\}
\)\\
& =
N(X)\sum_{\substack{d\leq D\\d\ \mathrm{squarefree}}}
\mu(d)\beta_{d^2}
+O\(X^{19/10+o(1)}+ D^{-1} X^{2+o(1)}\).
\end{split}
\end{equation}

Lemma~\ref{lem:beta-d-estimates} gives
\begin{equation}
\label{eq:Trunc Sum}
\sum_{\substack{d\leq D\\d\ \mathrm{squarefree}}}\mu(d)\beta_{d^2}
= \sum_{\substack{d\geq1\\d\ \mathrm{squarefree}}}\mu(d)\beta_{d^2} + O\(D^{-1+o(1)}\). 
\end{equation}
Since $\mu(d)$ and the function $\beta_{d^2}$ are
multiplicative,  we expand the last sum as the 
absolutely convergent Euler product
\begin{equation}
\label{eq:Euler Prod}
\sum_{\substack{d\geq1\\d\ \mathrm{squarefree}}}\mu(d)\beta_{d^2}
=\prod_p(1-\beta_{p^2})
=\mathfrak S_{\Rform}^{\mathrm{sq}}.
\end{equation}

Finally, $N(X)\ll X^2$ by~\eqref{eq:Krieg}.  Therefore, our choice of $D$,  
together with~\eqref{eq:mobius-truncation}, \eqref{eq:Trunc Sum} and~\eqref{eq:Euler Prod},
yields the desired result.

\section{Proof of Theorem~\ref{thm:almost-prime-lower-bound}} 

Put $Y=4X^2$ and define the nonnegative sequence
\[
a_X(n)=\#\{\gamma\in\Gamma_X:~\Rform(\gamma)=n\}.
\]
Thus $a_X(n)=0$ unless $2\leq n\leq Y$ and
\begin{align*}
\sum_{n\leq Y}a_X(n)&=N(X),\\
\sum_{\substack{n\leq Y\\d\mid n}}a_X(n)
&=A_d(X)=\beta_dN(X)+r_d(X),
\end{align*}
where $r_d(X)=A_d(X)-\beta_dN(X)$.

To establish the required level of distribution, fix $0<\vartheta<1/8$
and set $D=Y^\vartheta$.  For sufficiently large $X$
we have $D\leq X^{1/4}$.  Lemma~\ref{lem:circle-square-divisor-count}
therefore gives,
\begin{equation}
\label{eq:almost-prime-level}
\begin{split}
\sum_{d\leq D} |r_d(X)|
&\ll X^{19/10+o(1)}
       \sum_{d\leq D}\frac{1}{d^{3/5}}\\
&\leq X^{19/10+o(1)}D^{2/5}\\
&=X^{19/10+4\vartheta/5+o(1)}.
\end{split}
\end{equation}

We next verify the local conditions for the sieve.  The function
$d\mapsto\beta_d$ is multiplicative by its definition
in~\eqref{eq:beta-m-def}; in particular,
$\beta_d=\prod_{p\mid d}\beta_p$ for every squarefree $d\ge 1$.
Lemma~\ref{lem:local-square-density} gives
\[
\beta_2=\frac13 
\mand 
\beta_p=\frac{(p-\jac{-1}{p})^2}{p(p^2-1)}
=\frac1p+O\!\left(\frac1{p^2}\right), \quad p\ge 3.
\]
In particular, $0\leq\beta_p<1$ for every prime $p$.

Define
\[
V(z)=\prod_{p<z}(1-\beta_p).
\]
Mertens' theorem (see~\cite[Equation~(2.16)]{IwKow04}) and
the estimate $\beta_p=1/p+O(p^{-2})$ give
$V(z)\asymp1/\log z$ and
\[
\prod_{w\leq p<z}(1-\beta_p)^{-1}
=\frac{V(w)}{V(z)}
\leq\frac{\log z}{\log w}
\left(1+\frac{K}{\log w}\right)
\qquad (2\leq w<z),
\]
where $K$ is an absolute positive constant.

We use the multiset $\mathcal A$ in which $n$ occurs $a_X(n)$ times, with total
mass $N(X)$.  Note that in the notation of  Greaves~\cite[Section~1.3.1, Equation~(1.1)]{Greaves01},
the local function is $\rho(p)=p\beta_p$.
Let $r$ be an integer satisfying
\begin{equation}
\label{eq:almost-prime-sieve-criterion}
r>\frac1\vartheta+\frac{\log4}{\log3}-1.
\end{equation}
Take $3<s<4$ and $1<u<s$, with $s$ close to $4$ and $u$ close to $1$,
and put
\[
z_s=D^{1/s},\qquad \eta=r+1-\frac{u}{\vartheta},
\qquad P(v)=\prod_{p<v}p.
\]
We can assume that the parameters are chosen so that
\[
\eta > 1.
\]
Using the logarithmic weights as in~\cite[Section~5.1.1, Equation~(1.2)]{Greaves01},
for $\gcd(n,P(z_s))=1$ , we define the weight
\[
m(n)=\eta-
\sum_{\substack{p\mid n\\z_s\leq p<D^{1/u}}}
\left(1-\frac{u\log p}{\log D}\right).
\]
Since $n\leq Y=D^{1/\vartheta}$, we have 
\[
m(n)\leq\eta-\omega(n)+\frac{u\log n}{\log D}
\leq r+1-\omega(n).
\]
Thus a positive weight requires $\omega(n)\leq r$.
Applying~\cite[Section~5.1.1, Lemma~2]{Greaves01}
with 
\[
W(t)=
\begin{cases}
\eta-1+ut,&1/s\le t<1/u,\\
\eta,&1/u\le t\le1,
\end{cases}
\]
gives a lower bound for
\[
    \sum_{\substack{n \in \mathcal A\\ \gcd(n, P(z_s))=1}}m(n) = \sum_{\substack{n\leq Y\\\gcd(n,P(z_s))=1}}a_X(n)m(n).
\]
Indeed, using~\eqref{eq:almost-prime-level} to make the negative term in the lower bound of 
~\cite[Section~5.1.1, Lemma~2]{Greaves01} negligible under $0<\vartheta<1/8$,   we derive
\[
\sum_{\substack{n\leq Y\\\gcd(n,P(z_s))=1}}a_X(n)m(n)
\geq N(X)V(z_s)\(M(s,u)+o(1)\),
\]
where $M(s,u)$ is defined by~\cite[Section~5.1.1, Equation~(1.5)]{Greaves01} (as $M(W)$ for the 
above weight $W$). 
The dimension-one formulas for the linear sieve functions
in~\cite[Section~4.2.4, Equation~(4.16)]{Greaves01} give
\[
\lim_{\substack{s\to4^-\\u\to1^+}}M(s,u)
=\frac{e^{\gamma_{\mathrm{EM}}}}{2}
\left(\left(r+1-\frac1\vartheta\right)\log3-\log4\right)>0,
\]
where $\gamma_{\mathrm{EM}}= 0.5772\ldots$ is the Euler–Mascheroni constant and the positivity follows
from~\eqref{eq:almost-prime-sieve-criterion}.  We can therefore fix
$s$ and $u$ so that 
\[
\sum_{\substack{n\leq Y\\\gcd(n,P(z_s))=1}}a_X(n)m(n) \gg N(X)/\log Y.
\]

We recall that  $m(n) \le r+1$ and $m(n) \le 0$ if $\omega(n) \ge r+1$. 
So,
\begin{equation}
\label{eq:almost-prime-rough-count}
\#\{\gamma\in\Gamma_X:~\omega(\Rform(\gamma))\leq r\} \ge
\sum_{\substack{n\leq Y\\     \omega(n)\leq r\\  \gcd(n,P(z_s))=1}} a_X(n)
\gg\frac{N(X)}{\log Y}.
\end{equation}
We now take $\vartheta=3/25$ and $r=9$, for which~\eqref{eq:almost-prime-sieve-criterion} 
and thus~\eqref{eq:almost-prime-rough-count} hold.

It remains to remove repeated prime factors.  Set $z=Y^{\vartheta/4}$.
Every nonsquarefree value counted in the middle sum
in~\eqref{eq:almost-prime-rough-count} is divisible by $p^2$ for
some prime $p>z$.  By Lemma~\ref{lem:large-square-divisor-tail},
\[
\sum_{p>z} A_{p^2}(X)
\leq\sum_{d>z} A_{d^2}(X)
\leq z^{-1}X^{2+o(1)}
=X^{2-\vartheta/2+o(1)},
\]
which is negligible compared to the lower bound in~\eqref{eq:almost-prime-rough-count}.
Since $\log Y\asymp\log X$,
this concludes the proof.

\section{Comments}

One of the reasons for investigating $A_{m}(X)$ for arbitrary integers $m \ge 1$ (rather than only for squares of squarefree integers) is that such results on the level of distribution of $\Rform(\gamma)$ may have other applications. For example, another application is the study of smooth values of $\Rform(\gamma)$ for $\gamma\in \Gamma_X$. A natural question is whether one can take advantage of averaging over $m$ and get a higher level of distribution of $A_{m}(X)$ on average over $m \le M$.

\section*{Acknowledgements}

The authors are very grateful to Alisa Sedunova and Andreas Str{\"o}m\-bergsson for sending them  preliminary versions of~\cite{Sed} and~\cite{StrSoVi26}, respectively.

During the preparation of this work, I.S. was supported by the Australian Research Council
Grants DP230100530 and DP230100534 and also by 
the Max Planck Institute for Mathematics in Bonn, 
 where parts of this work have been carried out.  Y.X. acknowledges financial support from the China Scholarship Council and the National Natural Science Foundation of China (Grant Nos. 12171311 and 12671012). Y.X. also acknowledges the support and hospitality of the School of Mathematics and Statistics of the University of New South Wales through its PhD Support Scheme.


\begin{thebibliography}{99}



\bibitem{DiSh} 
F. Diamond and J. Shurman,  \emph{A first course in modular forms}, 
Grad. Texts in Math., vol.~228, Springer, 2005.

\bibitem{Folland99}
G. B. Folland,
\emph{Real Analysis: Modern Techniques and Their Applications},
2nd ed., John Wiley \& Sons, New York, 1999.

\bibitem{FrIw09}
J. B. Friedlander and H. Iwaniec,
Hyperbolic prime number theorem,
\emph{Acta Math.} \textbf{202} (2009), 1--19.

\bibitem{Greaves01}
G. Greaves,
\emph{Sieves in number theory},
Modern Surveys
in Math., vol.~43,
Springer-Verlag, Berlin, 2001.

\bibitem{HeathBrown96}
D. R. Heath-Brown,
A new form of the circle method, and its application to quadratic forms,
\emph{J. Reine Angew. Math.} \textbf{481} (1996), 149--206.

\bibitem{IwKow04}
H. Iwaniec and E. Kowalski,
\emph{Analytic number theory},
Amer.  Math.  Soc. Colloq. Publ., vol.~53,
Amer.  Math.  Soc., Providence, RI, 2004.

\bibitem{Krieg94}
A. Krieg,
Counting modular matrices with specified maximum norm,
\emph{Linear Algebra Appl.} \textbf{196} (1994), 273--278.

\bibitem{LN}
R. Lidl and H. Niederreiter,
\emph{Finite fields},
 Encycl. Math. and its Appls., vol.~20,
Cambridge Univ. Press, Cambridge, 1997.

\bibitem{MarmonVishe19}
O. Marmon and P. Vishe,
On the Hasse principle for quartic hypersurfaces,
\emph{Duke Math. J.} \textbf{168} (2019), 2727--2799.

\bibitem{Sed}
A. Sedunova,
On the  hyperbolic prime number theorem,
\emph{Preprint}, 2026.

\bibitem{StrSoVi26}
A. Str\"ombergsson, A. S\"odergren and P. Vishe,
Effective equidistribution of unipotent orbits in homogeneous spaces of
$\mathrm{SL}(2,\R)\ltimes(\R^2)^k$,
\emph{Preprint}, 2026 (available at \url{https://arxiv.org/abs/2604.08753}).



\end{thebibliography}
\end{document}